\documentclass[10pt]{article}%
\usepackage{amsmath}
\usepackage{amsfonts}
\usepackage{mathrsfs}
\usepackage{amssymb, color}
\usepackage[linkcolor=black,anchorcolor=black,citecolor=black]{hyperref}
\usepackage{graphicx}
\numberwithin{equation}{section}
\usepackage[body={15.5cm,21cm}, top=3cm]{geometry}%
\providecommand{\U}[1]{\protect\rule{.1in}{.1in}}
\providecommand{\U}[1]{\protect \rule{.1in}{.1in}}
\newtheorem{theorem}{Theorem}[section]

\newtheorem{definition}[theorem]{Definition}

\newtheorem{lemma}[theorem]{Lemma}

\newtheorem{proposition}[theorem]{Proposition}
\newtheorem{remark}[theorem]{Remark}

\newtheorem{assumption}[theorem]{Assumption}
\newenvironment{proof}[1][Proof]{\noindent \textbf{#1.} }{\  \rule{0.5em}{0.5em}}
\DeclareMathOperator*{\esssup}{ess\,sup}

\begin{document}
\title{Optimal Stopping under $G$-expectation}
\author{ Hanwu Li\thanks{Center for Mathematical Economics, Bielefeld University,
hanwu.li@uni-bielefeld.de. }
}
%\and Shige Peng\thanks{School of Mathematics and Qilu Institute of Finance, Shandong University,
%peng@sdu.edu.cn. Li and Peng's research was
%partially supported by the Tian Yuan Projection of the National Nature Sciences Foundation of China (No. 11526205 and No. 11626247)  and by the 111
%Project (No. B12023).}}
\date{}
\maketitle
\begin{abstract}
We develop a theory of optimal stopping problems under $G$-expectation framework. We first define a new kind of random times, called $G$-stopping times, which is suitable for this problem. For the discrete time case with finite horizon, the value function is defined backwardly and we show that it is the smallest $G$-supermartingale dominating the payoff process and the optimal stopping time exists. Then we extend this result both to the infinite horizon and to the continuous time case. We also establish the relation between the value function and solution of reflected BSDE driven by $G$-Brownian motion.
\end{abstract}

\textbf{Key words}: optimal stopping, $G$-expectation, $G$-stopping time, Knightian uncertainty%$G$-expectation, reflected backward stochastic differential equations, upper obstacle.

\textbf{MSC-classification}: 60H10, 60H30
\section{Introduction}
Consider a filtered probability space $(\Omega,\mathcal{F},P,\mathbb{F}=\{\mathcal{F}_t\}_{t\in[0,T]})$, the objective of optimal stopping problem is try to find a stopping time $\tau^*$ in order to maximize the expectation of $X_\tau$ over all stopping times. Here $X$ is a given progressively measurable and integrable process, called the payoff process. In financial market, $X$ can be regarded as the gain of an option. An agent has the right to stop this option at any time $t$ and then get the reward $X_t$, or to wait in the hope that he would obtain a bigger reward if he stops in the future. This problem has wide applications in finance and economics, such as pricing for American contingent claims and the decision of a firm to enter a new market.  Note that there is an implicit hypothesis in the above examples that the agent knows the probability distribution of the payoff process. This assumption excludes the case where the agent faces Knightian uncertainty. In this paper, we will investigate the optimal stopping problem under Knightian uncertainty, especially volatility uncertainty.

The optimal stopping problems under Knightian uncertainty attracks a great deal of attention due to its importance both in theory and in applications. We may refer to the papers \cite{BY1}, \cite{BY2}, \cite{CR}, \cite{ETZ}, \cite{NZ}, \cite{R}. Roughly speaking, Cheng and Riedel \cite{CR}, Riedel \cite{R} considered the optimal stopping problem in a multiple priors framework, which makes the linear expectation be a nonlinear one. Bayraktar and Yao \cite{BY1,BY2} studied this problem under what they called the filtration consistent nonlinear expectations. In these papers, they put assumptions either on the multiple priors $\mathcal{P}$ or on the nonlinear expectation $\mathcal{E}$ to make sure that the associated conditional expectation is time consistent and the optional sampling theorem still hold true. Similar with the classical case, the value function is an $\mathcal{E}$-supermartingale dominating the payoff process $X$. Besides, its first hitting time $\tau^*$ of $X$ is optimal and it is an $\mathcal{E}$-martingale up to time $\tau^*$. However, in the above papers, all probability measures in $\mathcal{P}$ are equivalent to a reference measure $P$ thus these models can only represent drift uncertainty. The ambiguity of volatility uncertainty requires that $\mathcal{P}$ is a family of non-dominated probability measures which makes this situation much more complicated. Ekren, Touzi and Zhang \cite{ETZ} and Nutz and Zhang \cite{NZ} investigated the optimal stopping problem under non-dominated family of probability measures. In fact, \cite{ETZ} studied the problem $\sup_{\tau}\sup_{P}E^P[X_\tau]$ which can be seen as a control problem while \cite{NZ} considered the problem $\inf_{\tau}\sup_{P}E^P[X_\tau]$ which can be regarded as a game problem. It is worth pointing out that in their papers, the value function is defined pathwise. They also obtained the optimality of the first hitting time $\tau^*$ and the nonlinear martingale property of the value function stopped at $\tau^*$.%  they only studied the finite horizon case and the value function is defined pathwise. %By using the properties of regular conditional probability distributions and analytic analysis, they can define the

According to the papers listed above, it is obvious that we need some nonlinear expectations to study the optimal stopping problem under Knightian uncertainty. Recently, Peng systematically established a time consistent nonlinear expectation theory, called $G$-expectation theory (see \cite{P07a, P08a}). As the counterpart of the classical linear expectation case, the notions of $G$-Brownian motion, $G$-martingale and $G$-It\^{o}'s integral were also introduced. A basic mathematical tool for the analysis is backward stochastic differential equations driven by $G$-Brownian motion ($G$-BSDE) studied by Hu, Ji, Peng and Song. In \cite{HJPS1,HJPS2}, they proved the existence and uniqueness of solutions to $G$-BSDE and established the corresponding comparison theorem, Girsanov transformation and Feynman-Kac formula.  The $G$-expectation theory is convincing as a useful tool for developing a theory of financial markets under volatility uncertainty. Therefore, the objective of this paper is to study the optimal stopping problem under the $G$-expectation framework. %motivated by the problems in fully nonlinear PDE and pricing

In the classical case, the value function is usually defined by the essential supremum over a set of random variables. However, in the $G$-framework, the essential supremum should be defined in the quasi-surely sense and thus may not exist. Besides, the random variables in $G$-framework require some continuity and monotonicity properties. For a random time $\tau$ and a process $X$, $X_\tau$ may not belong to a suitable space where the conditional $G$-expectation is well defined. Due to these difficulties lie in this problem, the optimal stopping under $G$-expectation is far from being understood.

In this paper, we first deal with optimal stopping problems under $G$-expectation in discrete time case, both in finite and infinite time horizon. We first restrict ourselves to a new kind of random time, called $G$-stopping time. The advantage is that the definition of conditional $G$-expectation can be extended to the random variable $X$ stopped at a  $G$-stopping time $\tau$. Besides, on this large space, many important properties, such as time consistency, still hold. For finite time case, we define the value function $V$ backwardly. It is not difficult to check that the value function is the minimal $G$-supermartingale dominating the payoff process $X$. Although we can only deal with a special kind of stopping times, we do not loose to much information since the first hitting time is an optimal $G$-stopping time. Then we extend the theory to infinite horizon case, where the backward induction cannot be applied directly. The value function is defined by the limit of the one in the finite time case. We show that it is still the minimal $G$-supermartingale which is greater than the payoff process and satisfies the recursive equations similar with the finite time case. Recall that Li et al. \cite{LPSH} studied the reflected BSDE driven by $G$-Brownian motion, which means the solution is required to be above an obstacle process. In fact, the solution of reflected $G$-BSDE is the minimal nonlinear supermartingale dominating the obstacle process. We show that it coincides with the value function of the optimal stopping problem in continuous time case when the payoff process $X$ equals to the obstacle process.

From the mathematical point of view, the optimal stopping problem introduced in \cite{ETZ} is the closest to ours. Compared with their results, the advantage of considering this problem under $G$-expectation lies in the following aspects. First, we do not need to assume the boundedness of the payoff process and we can study this problem when the time horizon is infinite. For the continuous time case, it can be shown that the value function is the limit of the ones of the discrete time case, which becomes useful to get numerical approximations. Besides, similar with the result in \cite{CR}, we can establish the relation between the value function defined by the Snell envelope and the solution of reflected BSDE driven by $G$-Brownian motion. At last, the case that the payoff process is Markovian can be involved and similar results as the classical case still hold.

This paper is organized as follows. In Section 2, we first introduce the $G$-stopping times and the essential supremum in the quasi-surely sense. Section 3 is devoted to study optimal stopping problems under $G$-expectation, both in finite and infinite time case. Then we extend the results to the continuous time case and we show that the value function of optimal stopping problem corresponds to the solution of reflected $G$-BSDE in Section 4. In Section 5, we presents some results of optimal stopping when the payoff process is Markovian.

\section{$G$-stopping times and essential supremum in the quasi-surely sense}

In this section, we introduce the essential supremum in the quasi-surely sense and a new kind of random time, called $G$-stopping time, appropriate for the study of optimal stopping under $G$-expectation. Then we investigate some properties of extended (conditional) $G$-expectation. Some basic notions and results of $G$-expectation can be found in the Appendix.%.% the extended conditional $G$-expectation and optional stopping theorem under $G$-framework. More details can be found in \cite{HP13}.

Let $(\Omega,L_G^1(\Omega),\hat{\mathbb{E}})$ be the $G$-expectation space and $\mathcal{P}$ be the weakly compact set that represents $\hat{\mathbb{E}}$. The following notations (see \cite{HP13}) will be frequently used in this paper.
\begin{align*}
L^0(\Omega)&:=\{X:\Omega\rightarrow [-\infty,\infty] \textrm{ and } X \textrm{ is } \mathcal{B}(\Omega)\textrm{-measurable}\},\\
\mathcal{L}(\Omega)&:=\{X\in L^0(\Omega):E^P[X] \textrm{ exists for each }P\in \mathcal{P}\},\\
\mathbb{L}^p(\Omega)&:=\{X\in L^0(\Omega):\hat{\mathbb{E}}[|X|^p]<\infty\} \textrm{ for } p\geq1,\\
L_G^{1^*}(\Omega)&:=\{X\in\mathbb{L}^1(\Omega):\exists \{X_n\}\subset L_G^1(\Omega) \textrm{ such that } X_n\downarrow X, q.s.\},\\
L_G^{*1}(\Omega)&:=\{X-Y:X,Y\in L_G^{1^*}(\Omega)\},\\
L_G^{1^*_*}(\Omega)&:=\{X\in\mathbb{L}^1(\Omega):\exists \{X_n\}\subset L_G^{1^*}(\Omega) \textrm{ such that } X_n\uparrow X, q.s.\},\\
\bar{L}_G^{1^*_*}(\Omega)&:=\{X\in\mathbb{L}^1(\Omega):\exists \{X_n\}\subset L_G^{1^*_*}(\Omega) \textrm{ such that } \hat{\mathbb{E}}[|X_n-X|]\rightarrow 0\}.
\end{align*}

\begin{remark}
	It is easy to check that $L_G^{1^*}(\Omega)\subset L_G^{*1}(\Omega)\subset L_G^{1^*_*}(\Omega)\subset \bar{L}_G^{1^*_*}(\Omega)$. Furthermore, it is important to note that $L_G^{1^*}(\Omega)$, $L_G^{1^*_*}(\Omega)$,  $\bar{L}_G^{1^*_*}(\Omega)$ are not linear spaces and $L_G^{*1}(\Omega)$ is a linear space.
\end{remark}

Set $\Omega_t=\{\omega_{\cdot\wedge t}:\omega\in\Omega\}$ for $t>0$. Similarly, we can define $L^0(\Omega_t)$, $\mathcal{L}(\Omega_t)$, $\mathbb{L}^p(\Omega_t)$, $L_G^{1^*}(\Omega_t)$, $L_G^{1^*_*}(\Omega_t)$ and $\bar{L}_G^{1^*_*}(\Omega_t)$ respectively. We now give the definition of stopping times under the $G$-expectation framework.
\begin{definition}
	A random time $\tau :\Omega\rightarrow [0,\infty)$ is called a $G$-stopping time if $I_{\{\tau\leq t\}}\in L_G^{1^*}(\Omega_t)$ for each $t\geq 0$.
\end{definition}

Let $\mathcal{H}\subset \mathbb{L}^1(\Omega)$ be a set of random variables. We give the definition of essential supremum of $\mathcal{H}$ in the quasi-surely sense. Roughly speaking, we only need to replace the ``almost-surely" in the classical definition by ``quasi-surely".
\begin{definition}\label{ess}
	The esssential supremum of $\mathcal{H}$, denoted by $\underset{\xi\in \mathcal{H}}{ess\sup}\text{ }\xi$, is a random variable in $L^0(\Omega)$ such that:
	\begin{description}
		\item[(i)] For any $\xi\in \mathcal{H}$,  $\underset{\xi\in \mathcal{H}}{ess\sup}\text{ }\xi\geq \xi$, q.s.
		\item[(ii)] If there exists another random variable $\eta'\in\mathcal{L}(\Omega)$ such that $\eta'\geq \xi$, q.s. for any $\xi\in\mathcal{H}$, then $\underset{\xi\in \mathcal{H}}{ess\sup}\text{ }\xi\leq \eta'$ q.s.
	\end{description}
\end{definition}

\begin{remark}
	It remains open to prove the existence of the essential supremum in the quasi-surely sense for the general case. However, if it does exist, it must be unique.
\end{remark}

\begin{remark}\label{esse2}	
	Consider a probability space $(\Omega, \mathcal{F}, P)$ and a set of random variables $\Phi$. Set
	\begin{displaymath}
	C:=\sup\{\sup_{\phi\in \widetilde{\Phi}} E^P[\phi]: \widetilde{\Phi}\textrm{ is a countable subset of } \Phi\}.
	\end{displaymath}
	There exists a countable set $\Phi^*:=\{\phi_n,n\in\mathbb{N}\}$ contained in $\Phi$, such that
	\begin{displaymath}
	C=\lim_{n\rightarrow \infty}E^P[\phi_n].
	\end{displaymath}
	Then $\eta:=\sup_{n\in \mathbb{N}}\phi_n$ is the essential supremum of $\Phi$ under $P$. However, this construction does not hold in the quasi-surely sense if we only replace the expectation $E^P$ by the $G$-expectation $\hat{\mathbb{E}}$. We may consider the following example.
	
		Let $1=\underline{\sigma}^2<\bar{\sigma}^2=2$. Consider $\mathcal{H}=\{f_\alpha(\langle B\rangle_1),\alpha\in[1,2]\}$, where $f_\alpha(x)=(x-1)^\alpha+1$. We can calculate that, for any $\alpha$,
	\[\hat{\mathbb{E}}[f_\alpha(\langle B\rangle_1)]=\sup_{x\in[1,2]}f_\alpha(x)=2.\]
	Then choose a countable subset $\mathcal{H}^*:=\{f_\alpha(\langle B\rangle_1),\alpha\in \mathbb{Q}\cap[\frac{3}{2},2]\}$. It is easy to check that
	\[\sup_{f_\alpha(\langle B\rangle_1)\in \mathcal{H}^*}\hat{\mathbb{E}}[f_\alpha(\langle B\rangle_1)]=\sup\{\sup_{\xi\in \widetilde{\mathcal{H}}} \hat{\mathbb{E}}[\xi]: \widetilde{\mathcal{H}}\textrm{ is a countable subset of } \mathcal{H}\}.\]
	However, $\sup_{f_\alpha(\langle B\rangle_1)\in \mathcal{H}^*}f_\alpha(\langle B\rangle_1)=f_{\frac{3}{2}}(\langle B\rangle_1)\leq f_1(\langle B\rangle_1)$ is not the essential supremum of $\mathcal{H}$.
\end{remark}
	
\begin{remark}
		 In the classical case, the essential supremum  can be constucted by countable many random variables while this does not hold true for the one in the quasi-surely sense. We may consider the following example. Let $1=\underline{\sigma}^2<\bar{\sigma}^2=2$. Consider $\mathcal{H}=\{I_{\{\langle B\rangle_1=x\}},x\in[1,2]\}$. If there exists $\widetilde{\mathcal{H}}=\{I_{\{\langle B\rangle_1=x_n\}}, x_n\in[1,2], n\in\mathbb{N}\}$, such that
	\begin{displaymath}
	\underset{\xi\in \mathcal{H}}{ess\sup}\textrm{ }\xi=\sup_{n\in \mathbb{N}} I_{\{\langle B\rangle_1=x_n\}}.
	\end{displaymath}
	Then there exists a constant $x_0\in[1,2]$ such that $x_0\neq x_n$, for any $n\in\mathbb{N}$. We have
	\begin{displaymath}
	c(\sup_{n\in \mathbb{N}} I_{\{\langle B\rangle_1=x_n\}}<I_{\{\langle B\rangle_1=x_0\}})=c(\langle B\rangle_1=x_0)=1,
	\end{displaymath}
	which is a contradiction.
\end{remark}

We now list some typcial situations under which the essential supremum exists.

\begin{proposition}\label{essp2}
	If there are only countable many random variables in $\mathcal{H}$, then the essential supremum exists.
\end{proposition}

\begin{proof}
	Without loss of generality, we may assume $\mathcal{H}=\{\xi_n, n\in\mathbb{N}\}$. We then define
	\begin{displaymath}
	\eta(\omega):=\sup_{n\in \mathbb{N}}\xi_n(\omega).
	\end{displaymath}
	We claim that $\eta$ is the essential supremum of $\mathcal{H}$. It is easy to check (i). We now show (ii) holds true. If $\eta'$ is a  random variable in $\mathcal{L}(\Omega)$ such that $\eta'\geq \xi_n$, q.s. for any $n\in \mathbb{N}$, then we have $c(\eta'<\xi_n)=0$, for any $n\in\mathbb{N}$. Note that
	\begin{displaymath}
	c(\eta'<\eta)=c(\{\omega: \exists n \textrm{ such that } \eta'<\xi_n\})\leq \sum_{n=1}^{\infty}c(\eta'<\xi_n)=0,
	\end{displaymath}
	which completes the proof.
\end{proof}

\begin{definition}
	A set $\widetilde{\mathcal{H}}$ is said to be dense in $\mathcal{H}$, if for any $\xi\in \mathcal{H}$, there exists a sequence $\{\xi_n, n\in\mathbb{N}\}\subset \widetilde{\mathcal{H}}$ such that $\mathbb{\hat{E}}[|\xi_n-\xi|]\rightarrow 0$ as $n\rightarrow\infty$.
\end{definition}

\begin{proposition}\label{essp3}
	If $\mathcal{H}$ has a countable dense subset, then the essential supremum of $\mathcal{H}$ exists.
\end{proposition}
\begin{proof}
	Without loss of generality, set $\widetilde{\mathcal{H}}:=\{\xi_m,m\in \mathbb{N}\}$ is the countable dense subset of $\mathcal{H}$. Denote
	\begin{displaymath}
	\eta(\omega):=\sup_{m\in\mathbb{N}}\xi_m(\omega).
	\end{displaymath}
	We claim that $\eta$ is the essential supremum of $\mathcal{H}$. It is sufficent to prove for any $\xi\in\mathcal{H}$, $\eta\geq \xi$, q.s. For any $\xi\in\mathcal{H}$, there exists a sequence $\{\hat{\xi}_n, n\in\mathbb{N}\}\subset \widetilde{\mathcal{H}}$ such that $\mathbb{\hat{E}}[|\hat{\xi}_n-\xi|]\rightarrow 0$ as $n\rightarrow\infty$. By Proposition 1.17 of Chapter VI in \cite{P10}, there exists a subsequence $\{\hat{\xi}_{n_k}\}_{k=1}^\infty$ such that,% which implies that $c(|\xi-\hat{\xi}_n|\geq \varepsilon)\rightarrow 0$ as $n\rightarrow \infty$ for any $\varepsilon>0$
	\begin{displaymath}
	\xi=\lim_{k\rightarrow\infty} \hat{\xi}_{n_k},\ \ q.s.
	\end{displaymath}
	Since for any $k$, $\hat{\xi}_{n_k}\leq \eta$, we have $\xi\leq \eta$, q.s.
\end{proof}

\begin{remark}
	Consider the example in Remark \ref{esse2}. Set $\widetilde{\mathcal{H}}:=\{f_\alpha(\langle B\rangle_1),\alpha\in[1,2]\cap \mathbb{Q}\}$. It is easy to check that $\widetilde{\mathcal{H}}$ is a countable dense subset of $\mathcal{H}$. Then $\eta:=\sup_{\xi\in\widetilde{\mathcal{H}}}\xi=f_1(\langle B\rangle_1)$ is the essential supremum of $\mathcal{H}$.
\end{remark}

\begin{proposition}\label{essp4}
	Assume that $\mathcal{H}\subset L_G^{1^*}(\Omega)$ is upwards directed and
	\begin{displaymath}
	\sup_{\xi\in\mathcal{H}}\hat{\mathbb{E}}[\xi]=\sup_{\xi\in\mathcal{H}}-\hat{\mathbb{E}}[-\xi].
	\end{displaymath}
	Then the essential supremum of $\mathcal{H}$ exists.
\end{proposition}

\begin{proof}
	Since the family $\mathcal{H}$ is upwards directed, there exist two increasing sequences $\{\xi^i_n,n\in\mathbb{N}\}\subset \mathcal{H}$, $i=1,2$, such that
	\begin{equation}\label{ess2}
	\lim_{n\rightarrow \infty}\hat{\mathbb{E}}[\xi^1_n]=\sup_{\xi\in\mathcal{H}}\hat{\mathbb{E}}[\xi],\ \
	\lim_{n\rightarrow \infty}-\hat{\mathbb{E}}[-\xi^2_n]=\sup_{\xi\in\mathcal{H}}-\hat{\mathbb{E}}[-\xi].
	\end{equation}
	We claim that $\eta:=\sup_{n}\eta_n$ is the essential supremum of $\mathcal{H}$, where $\eta_n=\xi^1_n \vee\xi_n^2$. Obviously, the second statement in Definition \ref{ess} holds true. We now prove the first statement. It is easy to check that
	\begin{align*}
	&\sup_{\xi\in\mathcal{H}}\hat{\mathbb{E}}[\xi]\geq \hat{\mathbb{E}}[\eta_n]\geq \hat{\mathbb{E}}[\xi^1_n],\\
	&\sup_{\xi\in\mathcal{H}}-\hat{\mathbb{E}}[-\xi]\geq -\hat{\mathbb{E}}[-\eta_n]\geq -\hat{\mathbb{E}}[-\xi^2_n].
	\end{align*}
	Letting $n\rightarrow\infty$, it follows that
	\begin{displaymath}
	\sup_{\xi\in\mathcal{H}}\hat{\mathbb{E}}[\xi]=\lim_{n\rightarrow\infty}\hat{\mathbb{E}}[\eta_n]
	=\lim_{n\rightarrow\infty}-\hat{\mathbb{E}}[-\eta_n]=\sup_{\xi\in\mathcal{H}}-\hat{\mathbb{E}}[-\xi].
	\end{displaymath}
	Applying Proposition 28 (7) in \cite{HP13}, we have
	\begin{displaymath}
	\hat{\mathbb{E}}[\eta]=\lim_{n\rightarrow\infty}\hat{\mathbb{E}}[\eta_n], \ \ -\hat{\mathbb{E}}[-\eta]=\lim_{n\rightarrow\infty}-\hat{\mathbb{E}}[-\eta_n],
	\end{displaymath}
	which implies that $\eta$ has no mean uncertainty. For any $\xi\in\mathcal{H}$, using the monotone convergence theorem, we get that
	\begin{displaymath}
	\hat{\mathbb{E}}[\eta\vee \xi]=\lim_{n\rightarrow\infty}\hat{\mathbb{E}}[\eta_n\vee\xi]\leq \sup_{\xi\in\mathcal{H}}\hat{\mathbb{E}}[\xi].
	\end{displaymath}
	Then we conclude that
	\begin{displaymath}
	0\leq \hat{\mathbb{E}}[\xi\vee \eta-\eta]=\hat{\mathbb{E}}[\xi\vee\eta]-\hat{\mathbb{E}}[\eta]\leq 0,
	\end{displaymath}
	which indicates that $\xi\vee \eta-\eta=0$, q.s. The proof is complete.
\end{proof}

In the following, we list some properties of the extended (conditional) $G$-expectation. It is natural to extend the definition of $G$-expectation $\hat{\mathbb{E}}$ to the space $\mathcal{L}(\Omega)$, still denoted by $\hat{\mathbb{E}}$. For each $X\in\mathcal{L}(\Omega)$, the extended $G$-expectation has the following representation
\[\hat{\mathbb{E}}[X]=\sup_{P\in\mathcal{P}}E^P[X].\]

\begin{lemma}\label{fatou}
	Let $\{X_n,n\in \mathbb{N}\}\subset \mathcal{L}(\Omega)$. Suppose that there exists a random variable $Y\in \mathcal{L}(\Omega)$ with $-\hat{\mathbb{E}}[-Y]>-\infty$ such that, for any $n\geq 1$, $X_n\geq Y$ q.s. Then $\liminf_{n\rightarrow \infty}X_n\in \mathcal{L}(\Omega)$ and
	\begin{displaymath}
		\hat{\mathbb{E}}[\liminf_{n\rightarrow \infty}X_n]\leq \liminf_{n\rightarrow \infty}\hat{\mathbb{E}}[X_n].
	\end{displaymath}
\end{lemma}

\begin{proof}
	By the classical monotone convergence theorem and Fatou's Lemma, we have for each $P\in \mathcal{P}$, $E^P[\liminf_{n\rightarrow \infty}X_n]$ exists  and
	\begin{displaymath}
		E^P[\liminf_{n\rightarrow \infty}X_n]\leq \liminf_{n\rightarrow \infty}E^P[X_n]\leq \liminf_{n\rightarrow \infty}\hat{\mathbb{E}}[X_n].
	\end{displaymath}
	Taking supremum over all $P\in \mathcal{P}$, we get the desired result.
\end{proof}

\begin{remark}
	It is worth pointing out that the Fatou Lemma of the``$\limsup$" type does not hold under $G$-expectation. For example, set $0<\underline{\sigma}^2<\bar{\sigma}^2=1$. Consider the sequence $\{X_n,n\in\mathbb{N}\}$, where $X_n=I_{\{\langle B\rangle_1 \in(1-\frac{1}{n},1)\}}$. It is easy to check that $X_n\rightarrow 0$ and $\hat{\mathbb{E}}[X_n]=1$ for any $n\in\mathbb{N}$. Therefore, we have
	\begin{displaymath}
		0=\hat{\mathbb{E}}[\limsup_{n\rightarrow \infty}X_n]<\limsup_{n\rightarrow \infty}\hat{\mathbb{E}}[X_n]=1.
	\end{displaymath}
\end{remark}

In fact, for any given $X\in \bar{L}_G^{1^*_*}(\Omega)$, the supremum of expectation over all probability $P\in\mathcal{P}$ is attained.
\begin{proposition}\label{os22}
	For any $X\in \bar{L}_G^{1^*_*}(\Omega)$, there exists some $P\in \mathcal{P}$, such that
	\begin{displaymath}
	\hat{\mathbb{E}}[X]=E^{P}[X].
	\end{displaymath}
\end{proposition}

\begin{proof}
	We first claim that for any $X\in L_G^{1^*_*}(\Omega)$, if $P_n\rightarrow P$ weakly, then we have
	\begin{equation}\label{re27}
	\limsup_{n\rightarrow\infty}E^{P_n}[X]\leq E^P[X].
	\end{equation}
	We first prove Equation \eqref{re27} for any $X\in L_G^{1^*}(\Omega)$. Note that there exists a sequence of random variables $\{X_m\}\subset L_G^1(\Omega)$ such that $X_m\downarrow X$, q.s. Then for any $m\in\mathbb{N}$, we have $E^{P_n}[X]\leq E^{P_n}[X_m]$. Applying Lemma 1.29 of Chapter VI in \cite{P10}, it follows that% which implies that
	\begin{displaymath}
	\limsup_{n\rightarrow\infty}E^{P_n}[X]\leq \limsup_{n\rightarrow\infty}E^{P_n}[X_m]=E^P[X_m].
	\end{displaymath}
	Letting $m\rightarrow\infty$, by the monotone convergence theorem, Equation \eqref{re27} holds. For any $X\in L_G^{1^*_*}(\Omega)$, there exists some $\{X_m\}\subset L_G^{1^*}(\Omega)$, such that $X_m\uparrow X$, q.s. By monotone convergence theorem, it is easy to check that
	\begin{align*}
	\limsup_{n\rightarrow\infty}E^{P_n}[X]&=\limsup_{n\rightarrow\infty}\liminf_{m\rightarrow \infty}E^{P_n}[X_m]\leq \liminf_{m\rightarrow \infty}\limsup_{n\rightarrow\infty}E^{P_n}[X_m]\\
	&\leq \liminf_{m\rightarrow \infty}E^P[X_m]=E^{P}[X].
	\end{align*}
	Now for any $X\in \bar{L}_G^{1^*_*}(\Omega)$, there exists a sequence of random variables $\{X_m\}\in {L}_G^{1^*_*}(\Omega)$ such that
	\begin{displaymath}
	\lim_{m\rightarrow \infty}\sup_{P\in\mathcal{P}}E^P{|X_m-X|}=\lim_{m\rightarrow \infty}
	\hat{\mathbb{E}}[|X_m-X|]=0.
	\end{displaymath}
	Then for any $\varepsilon>0$, there exists some $M$ independent of $P$, such that, for any $m\geq M$, $E^P[X]\leq E^P[X_m]+\varepsilon$. By the definition of extended $G$-expectation, we can choose a sequence of probability measures $\{P_n\}$, such that
	\begin{displaymath}
	\hat{\mathbb{E}}[X]=\lim_{n\rightarrow \infty}E^{P_n}[X].
	\end{displaymath}
	Noting that $\mathcal{P}$ is weakly compact, without loss of generality, we may assume that $P_n$ converges to $P$ weakly. We can calculate that
	\begin{displaymath}
	\hat{\mathbb{E}}[X]=\lim_{n\rightarrow \infty}E^{P_n}[X]\leq \limsup_{n\rightarrow \infty}E^{P_n}[X_m]+\varepsilon\leq E^P[X_m]+\varepsilon\rightarrow E^P[X]+\varepsilon, \textrm{ as } m\rightarrow\infty.	
	\end{displaymath}
	Since $\varepsilon$ can be arbitrarily small, we get the desired result.
\end{proof}

Now we extend the definition of conditional $G$-expectation. For this purpose, we need the following lemma, which generalizes Lemma 2.4 in \cite{HJPS1}.
\begin{lemma}\label{os7}
	For each $\xi,\eta\in L_G^{*1}(\Omega)$ and $A\in \mathcal{B}(\Omega_t)$, if $\xi I_A\geq \eta I_A$ q.s., then $\hat{\mathbb{E}}_t[\xi]I_A\geq \hat{\mathbb{E}}_t[\eta]I_A$ q.s.
\end{lemma}
\begin{proof}
	Otherwise, we may choose a compact set $K\subset A$ with $c(K)>0$ such that $(\hat{\mathbb{E}}_t[\xi]-\hat{\mathbb{E}}_t[\eta])^->0$ on $K$. Noting that $K$ is compact, there exists a sequence of nonnegative functions $\{\zeta_n\}_{n=1}^\infty\subset C_b(\Omega_t)$ such that $\zeta_n\downarrow I_K$, which implies that $I_K\in L_G^{1^*}(\Omega_t)$. Since $\xi,\eta\in L_G^{*1}(\Omega)$, there exist $\xi_i,\eta_i\in L_G^{1^*}(\Omega)$ and $\{\xi_i^n\}_{n=1}^\infty,\{\eta_i^n\}_{n=1}^\infty \subset L_G^1(\Omega)$, $i=1,2$ such that $\xi_i^n\downarrow \xi$, $\eta_i^n\downarrow \eta$ and
	\begin{displaymath}
	\xi=\xi_1-\xi_2, \ \ \eta=\eta_1-\eta_2.
	\end{displaymath}
	Set $X_n=\xi_1^n+\eta_2^n$, $Y_n=\xi_2^n+\eta_1^n$. Then $\{X_n\}_{n=1}^\infty,\{Y_n\}_{n=1}^\infty\subset L_G^1(\Omega)$ and they are decreasing in $n$. We denote by $X,Y$ the limit of  $\{X_n\}_{n=1}^\infty,\{Y_n\}_{n=1}^\infty$ respectively. It is easy to check that $X,Y\in L_G^{1^*}(\Omega)$ and $\xi-\eta=X-Y$. For each fixed $l,m,n\in \mathbb{N}$, we have
	\begin{displaymath}
	\hat{\mathbb{E}}[\zeta_l(X_n-Y_m)^-]\downarrow\hat{\mathbb{E}}[I_K(X_n-Y_m)^-], \textrm{ as } l\rightarrow\infty,
	\end{displaymath}
	and
	\begin{displaymath}
	\hat{\mathbb{E}}[\zeta_l\hat{\mathbb{E}}_t[(X_n-Y_m)^-]]\downarrow\hat{\mathbb{E}}[I_K\hat{\mathbb{E}}_t[(X_n-Y_m)^-]], \textrm{ as } l\rightarrow\infty,
	\end{displaymath}
	Noting that
	\begin{displaymath}
	\hat{\mathbb{E}}[\zeta_l(X_n-Y_m)^-]= \hat{\mathbb{E}}[\zeta_l\hat{\mathbb{E}}_t[(X_n-Y_m)^-]],
	\end{displaymath}
	it follows that
	\begin{displaymath}
	\hat{\mathbb{E}}[I_K(X_n-Y_m)^-]=\hat{\mathbb{E}}[I_K\hat{\mathbb{E}}_t[(X_n-Y_m)^-]].
	\end{displaymath}
	For each fixed $m,n\in \mathbb{N}$, we have $I_K(X_n-Y_m)^-\in L_G^{1^*}(\Omega)$. First letting $m\rightarrow\infty$, we obtain $I_K(X_n-Y_m)^-\downarrow I_K(X_n-Y)^-$ and $I_K(X_n-Y)^-\in L_G^{1^*}(\Omega)$. Then letting $n\rightarrow\infty$, we obtain $I_K(X_n-Y)^-\uparrow I_K(X-Y)^-$ and $I_K(X-Y)^-\in L_G^{1^*_*}(\Omega)$. Therefore, we can calculate that
	\begin{displaymath}
	\lim_{n\rightarrow \infty}\lim_{m\rightarrow \infty}\hat{\mathbb{E}}[I_K(X_n-Y_m)^-]=\hat{\mathbb{E}}[I_K(X-Y)^-]=\hat{\mathbb{E}}[I_K(\xi-\eta)^-]=0.
	\end{displaymath}
	By a similar analysis, we have
	\begin{displaymath}
	\lim_{n\rightarrow \infty}\lim_{m\rightarrow \infty}\hat{\mathbb{E}}[I_K\hat{\mathbb{E}}_t[(X_n-Y_m)^-]]=\hat{\mathbb{E}}[I_K\hat{\mathbb{E}}_t[(X-Y)^-]]=\hat{\mathbb{E}}[I_K\hat{\mathbb{E}}_t[(\xi-\eta)^-]].
	\end{displaymath}
	By Proposition 34 (4) and (8) in \cite{HP13}, we can check that
	\begin{displaymath}
	(\hat{\mathbb{E}}_t[\xi]-\hat{\mathbb{E}}_t[\eta])^-\leq \hat{\mathbb{E}}_t[(\xi-\eta)^-],
	\end{displaymath}
	which yields $\hat{\mathbb{E}}_t[(\xi-\eta)^-]>0$ on $K$. Recall that $c(K)>0$ which implies that $\hat{\mathbb{E}}[I_K\hat{\mathbb{E}}_t[(\xi-\eta)^-]]>0$. This is a contradiction and the proof is complete.
\end{proof}

Lemma \ref{os7} allows us to extend the definition of conditional $G$-expectation. For each $t\geq 0$, set
\begin{displaymath}
L_G^{*1,0,t}(\Omega)=\{\xi=\sum_{i=1}^n \eta_i I_{A_i}:\{A_i\}_{i=1}^n \textrm{ is a partition of } \mathcal{B}(\Omega_t), \eta_i\in L_G^{*1}(\Omega),n\in \mathbb{N}\}.
\end{displaymath}

\begin{definition}
For each $\xi \in L_G^{*1,0,t}(\Omega)$ with representation $\xi=\sum_{i=1}^n \eta_i I_{A_i}$, we define the conditional expectation, still denoted by $\hat{\mathbb{E}}_s$, by setting
\begin{equation}\label{re15}
\hat{\mathbb{E}}_s[\xi]:=\sum_{i=1}^n \hat{\mathbb{E}}_s[\eta_i] I_{A_i}, \textrm{ for } s\geq t.
\end{equation}
\end{definition}

\begin{remark}
	If furthermore, $\xi\in \bar{L}_G^{1^*_*}(\Omega)\cap L_G^{*1,0,t}(\Omega)$, then the extended conditional $G$-expectation \eqref{re15} coincides with the one as in Proposition \ref{os19}. In fact, for any $\xi \in L_G^{*1,0,t}(\Omega)$ with representation $\xi=\sum_{i=1}^n \eta_i I_{A_i}$ and $s\geq t$, we can calculate that
	\begin{align*}
	\sum_{i=1}^n \hat{\mathbb{E}}_s[\eta_i] I_{A_i}=& \sum_{i=1}^n I_{A_i} {\esssup_{Q\in\mathcal{P}(s,P)}}^P E^Q[\eta_i|\mathcal{F}_s]
	={\esssup_{Q\in\mathcal{P}(s,P)}}^P (\sum_{i=1}^n E^Q[\eta_i|\mathcal{F}_s] I_{A_i})\\
	=&{\esssup_{Q\in\mathcal{P}(s,P)}}^P  E^Q[\sum_{i=1}^n\eta_i I_{A_i}|\mathcal{F}_s] ={\esssup_{Q\in\mathcal{P}(s,P)}}^P  E^Q[\xi|\mathcal{F}_s],\ \textrm{P-}a.s.,
	\end{align*}	
	for any $P\in\mathcal{P}$.
\end{remark}

\section{Optimal stopping in discrete-time case}
In this section, we study the optimal stopping problem under $G$-expectation for the discrete time case, i.e. the $G$-stopping time $\tau$ takes values in some discrete set. We first investigate the finite time case by applying the method of backward induction and then extend the results to the infinite time case.
\subsection{Finite time horizon case}

In this subsection, we need to assume the payoff process $X$ satisfies the following condition.
\begin{assumption}\label{a1}
	$\{X_n, n=0,1,\cdots, N\}$ is a sequence of random variables such that for any $n$, $X_n\in L_G^1(\Omega_n)$.% and $X_n$ is bounded from below.
\end{assumption}

\begin{theorem}\label{os1}
	Let Assumption \ref{a1} hold true. We define the following sequence $\{V_n,n=0,1,\cdots,N\}$ by backward induction: Let $V_N=X_N$ and
	\begin{displaymath}
	V_n=\max\{X_n,\hat{\mathbb{E}}_n[V_{n+1}]\}, \ n\leq N-1.
	\end{displaymath}
	Then we have the following conclusion:
	\begin{description}
		\item[(1)] $\{V_n,n=0,1,\cdots,N\}$ is the smallest $G$-supermartingale dominating $\{X_n,n=0,1,\cdots,N\}$;
		\item[(2)] Denote by $\mathcal{T}_{j,N}$ the set of all $G$-stopping time taking values in $\{j,\cdots,N\}$. Set $\tau_{j}=\inf\{l\geq j: V_l=X_l\}$. Then $\tau_{j}$ is a $G$-stopping time and $V_{n\wedge \tau_j}\in L_G^{*1}(\Omega_n)$, for any $j\leq N$ and $n\leq N$. Furthermore, $\{V_{n\wedge \tau_j},n=0,1,\cdots,N\}$ is a $G$-martingale and for any $j\leq N$,
		\begin{displaymath}
		V_j=\hat{\mathbb{E}}_j[X_{\tau_j}]=\underset{\tau\in \mathcal{T}_{j,N}}{ess\sup}\textrm{ }\hat{\mathbb{E}}_j[X_\tau].
		\end{displaymath}
	\end{description}
\end{theorem}

\begin{proof}
%	For simplicity, we assume that $X_n\geq 0$, $n=0,1,\cdots,N$.
	\textbf{(1)} It is easy to check that for any $n=0,1,\cdots,N$,
	\begin{displaymath}
	V_n\geq X_n, \textrm{ and } V_n\geq \hat{\mathbb{E}}_{n}[V_{n+1}],
	\end{displaymath}
	which implies $\{V_n, n=0,1,\cdots, N\}$ is a $G$-supermartingale dominating $\{X_n, n=0,1,\cdots, N\}$. If $\{U_n, n=0,1,\cdots, N\}$ is another $G$-supmartingale dominating $\{X_n, n=0,1,\cdots, N\}$, we have $U_N\geq X_N=V_N$ and
	\begin{displaymath}
	U_{N-1}\geq \hat{\mathbb{E}}_{N-1}[U_{N}]\geq \hat{\mathbb{E}}_{N-1}[V_{N}], \ U_{N-1}\geq X_{N-1}.
	\end{displaymath}
	It follows that $U_{N-1}\geq V_{N-1}$. By induction, we can prove that for all $n=0,1,\cdots,N$, $V_n\leq U_n$.
	
	\textbf{(2)} For any $n=j,\cdots,N$, we can check that
		\begin{displaymath}
	\{\tau_j\leq n\}=\cup_{k=j}^n \{\tau_{k}=k\}=\cup_{k=j}^n\{V_k=X_k\}.%=\cup_{k=j}^n \{d(V_k-X_k,0)=0
	\end{displaymath}
	and
	\begin{displaymath}
	I_{\{\tau_j\leq n\}}=\max_{j\leq k\leq n} I_{\{V_k-X_k=0\}}.
	\end{displaymath}
	By Remark \ref{o8}, $I_{\{V_k-X_k=0\}}\in L_G^{1^*}(\Omega_k)$, for any $j\leq k\leq n$. It follows that $I_{\{\tau_j\leq n\}}\in  L_G^{1^*}(\Omega_n)$.% we may assume that they are continuous. It follows that $\{\tau_{j}\leq n\}$ is a closed subset in $\Omega_n$. There exists a sequence of bounded continuous functions $\{\zeta_n,n\in\mathbb{N}\}\subset L_{G}^1(\Omega_n)$ such that $\zeta_n\downarrow I_{\{\tau_j\leq n\}}$, which implies that .
	
	It is easy to check that, for any $j<n\leq N$,
	\begin{equation}\label{re5}
	\begin{split}
	V_{n\wedge \tau_j}&=\sum_{k=j}^{n-1}(V_k-V_{k+1}) I_{\{\tau_j\leq k\}}+V_n\\
	&=\sum_{k=j}^{n-1}(V_k-V_{k+1})^+ I_{\{\tau_j\leq k\}}-\sum_{k=j}^{n-1}(V_k-V_{k+1})^- I_{\{\tau_j\leq k\}}+V_n
	.
	\end{split}
	\end{equation}
	We conclude that $V_{n\wedge \tau_j}\in L_G^{*1}(\Omega_n) $. Note that
	\begin{equation}\label{re6}
	V_{(n+1)\wedge \tau_j}-V_{n\wedge \tau_j}= I_{\{\tau_j\geq n+1\}}(V_{n+1}-V_n)=I_{\{\tau_j\leq n\}^c}(V_{n+1}-\hat{\mathbb{E}}_n[V_{n+1}]).
	\end{equation}
	Since $\{\tau_j\leq n\}\in \mathcal{B}(\Omega_n)$, applying Lemma \ref{os7} and equation \eqref{re15}, we have
	\begin{equation}\label{re7}
	\hat{\mathbb{E}}_n[I_{\{\tau_j\leq n\}^c}(V_{n+1}-\hat{\mathbb{E}}_n[V_{n+1}])]=I_{\{\tau_j\leq n\}^c}\hat{\mathbb{E}}_n[(V_{n+1}-\hat{\mathbb{E}}_n[V_{n+1}])]=0.
	\end{equation}
	Combining \eqref{re6} and \eqref{re7}, we can get
	\begin{displaymath}
	0=\hat{\mathbb{E}}_n[V_{(n+1)\wedge \tau_j}-V_{n\wedge \tau_j}]=\hat{\mathbb{E}}_n[V_{(n+1)\wedge \tau_j}]-V_{n\wedge \tau_j},
	\end{displaymath}
	which shows that $\{V_{n\wedge \tau_j},n=j,j+1,\cdots,N\}$ is a $G$-martingale. Consequently, we have
	\begin{displaymath}
	V_j=\hat{\mathbb{E}}_j[V_{\tau_j}]=\hat{\mathbb{E}}_j[X_{\tau_j}].
	\end{displaymath}
	
	We claim that, for any $\tau\in \mathcal{T}_{j,N}$,
	\begin{equation}\label{re8}
	V_j\geq \hat{\mathbb{E}}_j[X_\tau].
	\end{equation}
	First, similar with \eqref{re5}, we obtain $X_\tau\in L_G^{*1}(\Omega_N)$. We then calculate that
	\begin{align*}
	\hat{\mathbb{E}}_{N-1}[X_\tau]&\leq \hat{\mathbb{E}}_{N-1}[V_\tau]=\hat{\mathbb{E}}_{N-1}[\sum_{k=j}^{N-1}(V_k-V_{k+1}) I_{\{\tau\leq k\}}+V_N]\\
	&=\sum_{k=j}^{N-2}(V_k-V_{k+1}) I_{\{\tau\leq k\}}+V_{N-1}I_{\{\tau\leq N-1\}}+\hat{\mathbb{E}}_{N-1}[V_N I_{\{\tau=N\}}]\\
	&=\sum_{k=j}^{N-2}(V_k-V_{k+1}) I_{\{\tau\leq k\}}+V_{N-1}I_{\{\tau\leq N-1\}}+\hat{\mathbb{E}}_{N-1}[V_N ]I_{\{\tau=N\}}\\
	&\leq \sum_{k=j}^{N-2}(V_k-V_{k+1}) I_{\{\tau\leq k\}}+V_{N-1}=V_{(N-1)\wedge \tau},
	\end{align*}
	where we use Equation \eqref{re15} again in the last equality. Repeat this procedure, we get that \eqref{re8} holds. The proof is complete.
\end{proof}

\begin{remark}
	If $\{V_n\}_{n=0}^N$ is defined by
	\begin{displaymath}
    V_N=X_N, \ \ 	V_n=\min\{X_n,\hat{\mathbb{E}}_n[V_{n+1}]\}, \ n\leq N-1.
	\end{displaymath}
	By a similar analysis, we have the following conclusion:
	\begin{description}
		\item[(1)] $\{V_n,n=0,1,\cdots,N\}$ is the largest $G$-submartingale dominated by $\{X_n,n=0,1,\cdots,N\}$;
		\item[(2)] Set $\tau_{j}=\inf\{l\geq j: V_l=X_l\}$. Then $\tau_{j}$ is a $G$-stopping time and $V_{n\wedge \tau_j}\in L_G^{*1}(\Omega_n)$, for any $j\leq N$ and $n\leq N$. Furthermore, $\{V_{n\wedge \tau_j},n=0,1,\cdots,N\}$ is a $G$-martingale and for any $j\leq N$,
		\begin{displaymath}
		V_j=\hat{\mathbb{E}}_j[X_{\tau_j}]=\underset{\tau\in \mathcal{T}_{j,N}}{ess\inf}\textrm{ }\hat{\mathbb{E}}_j[X_\tau].
		\end{displaymath}
	\end{description}
	By Proposition \ref{os22}, there exists some $P\in\mathcal{P}$ such that \[V_0=\inf_{\tau\in \mathcal{T}_{0,N}}\sup_{P\in\mathcal{P}}E^P[X_\tau]=E^P[X_{\tau_0}].\]
\end{remark}

\subsection{Infinite time horizon case}

Now we study the infinite time case. The conditions on the payoff process are more restrictive compared with the finite time case mainly due to the fact that the order of the right-hand side of Doob's inequality under $G$-expectation is strictly larger than the one of the left-hand side.
\begin{assumption}\label{a2}
	$\{X_n,n\in\mathbb{N}\}$ is a sequence of random variables bounded from below and for any $n\in\mathbb{N}$, $X_n\in L_G^\beta(\Omega_n)$, where $\beta>1$. Furthermore,
	\begin{displaymath}
	\hat{\mathbb{E}}[\sup_{n\in \mathbb{N}}|X_n|^\beta]<\infty.
	\end{displaymath}
\end{assumption}

For each fixed $N\in\mathbb{N}$, we define the following sequence $\{\tilde{V}^N_n,n=0,1,\cdots,N\}$ by backward induction: Let $\tilde{V}^N_N=X_N$ and
\begin{equation}\label{re12}
\tilde{V}_n^N=\max\{X_n,\hat{\mathbb{E}}_n[\tilde{V}^N_{n+1}]\}, \ n\leq N-1.
\end{equation}
It is easy to check that for any $n\leq N\leq M$, $\tilde{V}_n^N\leq \tilde{V}_n^M$. We may define
\begin{equation}\label{re13}
\tilde{V}^\infty_n=\lim_{N\geq n,N\rightarrow \infty}\tilde{V}_n^N.
\end{equation}

\begin{proposition}\label{os5}
	The sequence $\{\tilde{V}_n^\infty,n\in\mathbb{N}\}$ defined by \eqref{re13} is the smallest $G$-supermartingale dominating the process $\{X_n,n\in\mathbb{N}\}$.
\end{proposition}
\begin{proof}
	By monotone convergence theorem, letting $N\rightarrow\infty$ in \eqref{re12}, we have
	\begin{displaymath}
	\tilde{V}_n^\infty=\max\{X_n,\hat{\mathbb{E}}_n[\tilde{V}_{n+1}^\infty]\},
	\end{displaymath}
	which implies that $\{\tilde{V}_n^\infty,n\in\mathbb{N}\}$  is a $G$-supermartingale dominating the process $\{X_n,n\in\mathbb{N}\}$. Let $\{U_n,n\in\mathbb{N}\}$ be a $G$-supermartingale dominating the process $\{X_n,n\in\mathbb{N}\}$. By Theorem \ref{os1}, $\{\tilde{V}_n^N,n=0,1,\cdots,N\}$ is the smallest $G$-supermartingale dominating $\{X_n,n=0,1,\cdots,N\}$. Then for each $n\in\mathbb{N}$ and $N\geq n$, we have $\tilde{V}_n^N\leq U_n$. It follows that
	\begin{displaymath}
	U_n\geq \lim_{N\rightarrow \infty}\tilde{V}_n^N=\tilde{V}_n^\infty,
	\end{displaymath}
	which yields that $\{\tilde{V}_n^\infty,n=0,1,\cdots,N\}$ is the smallest $G$-supermartingale dominating $\{X_n,n=0,1,\cdots,N\}$.
\end{proof}

For each $j\in \mathbb{N}$, denote by $\mathcal{T}_{j}$ the collection of all $G$-stopping times taking values in $\{j,j+1,\cdots\}$ such that
\begin{equation}\label{re14}
\lim_{N\rightarrow \infty} c(\tau>N)=0.
\end{equation}
Set
\begin{displaymath}
\tilde{V}_0=\sup_{\tau\in\mathcal{T}_{0}}\hat{\mathbb{E}}[X_\tau].
\end{displaymath}

\begin{remark}
If a $G$-stopping time $\tau$ satisfies condition \eqref{re14}, noting that $\{\tau=\infty\}\subset \{\tau>N\}$ for any $N\in\mathbb{N}$, we obtain that
\begin{displaymath}
0\leq c(\tau=\infty)\leq \lim_{N\rightarrow \infty}c(\tau>N)=0,
\end{displaymath}
which implies that $\tau$ is finite quasi-surely. However, the inverse does not hold true. Consider the following example. Let $0<\underline{\sigma}^2<\bar{\sigma}^2=1$. Set
\begin{displaymath}
\tau=\begin{cases}
1, &\textrm{ if } \langle B\rangle_1=0,\\
N, &\textrm{ if } \langle B\rangle_1\in(1-\frac{1}{N-1},1-\frac{1}{N}], N\geq 2.
\end{cases}
\end{displaymath}
It is easy to check that $\tau$ is a $G$-stopping time and $c(\tau=\infty)=0$. However, for any fixed $N\in\mathbb{N}$, we have
\begin{displaymath}
c(\tau>N)=c(\langle B\rangle_1>1-\frac{1}{N})=1.
\end{displaymath}
\end{remark}

\begin{proposition}\label{os6}
	Under the above assumptions, we have
	\begin{displaymath}
	\tilde{V}_0=\tilde{V}_0^\infty.
	\end{displaymath}
\end{proposition}

\begin{proof}
	By Theorem \ref{os1}, it is obvious that $\tilde{V}_0\geq \tilde{V}_0^N$, for any $N\in\mathbb{N}$, which implies that $\tilde{V}_0\geq \tilde{V}_0^\infty$. We then prove the inverse inequality. For any $\tau\in \mathcal{T}_0$ and $\varepsilon>0$, there exists some $N$ such that  $c(\tau>N)\leq \varepsilon$. By Assumption 3.4 and the H\"{o}lder inequality, we can calculate that
	\begin{equation}\label{re}
	\hat{\mathbb{E}}[|X_\tau-X_{\tau \wedge N}|]\leq \hat{\mathbb{E}}[2\sup_{n\in\mathbb{N}}|X_n|I_{\{\tau>N\}}]\leq C(\hat{\mathbb{E}}[\sup_{n\in\mathbb{N}}|X_n|^\beta])^{\frac{1}{\beta}}
(\hat{\mathbb{E}}[I_{\{\tau>N\}}])^{\frac{\beta-1}{\beta}}\leq C\varepsilon^{\frac{\beta-1}{\beta}}.
	\end{equation}
	It follows that
	\begin{displaymath}
	\hat{\mathbb{E}}[X_\tau]\leq \hat{\mathbb{E}}[X_{\tau\wedge N}]+C\varepsilon^{\frac{\beta-1}{\beta}}\leq \tilde{V}_0^N+C\varepsilon^{\frac{\beta-1}{\beta}}\leq \tilde{V}_0^\infty+C\varepsilon^{\frac{\beta-1}{\beta}}.
	\end{displaymath}
	Letting $\varepsilon\rightarrow\infty$, since $\tau$ is arbitrarily chosen, we finally get the desired result.
\end{proof}
%Then we have the following conclusion.
%\begin{description}
%	\item[(1)] $\{V_n,n=0,1,\cdots,N\}$ is the smallest $G$-supermartingale dominating $\{X_n,n=0,1,\cdots,N\}$;
%	\item[(2)] Denote by $\mathcal{T}_{j,N}$ the set of all $G$-stopping time taking values in $\{j,\cdots,N\}$. Set $\tau_{j}=\inf\{l\geq j: V_l=X_l\}$. Then $\tau_{j}$ is a $G$-stopping time and $V_{n\wedge \tau_j}\in L_G^{*1}(\Omega_n)$, for any $j\leq N$ and $n\leq N$. Furthermore, $\{V_{n\wedge \tau_j},n=0,1,\cdots,N\}$ is a $G$-martingale and for any $j\leq N$,
%	\begin{displaymath}
%	V_j=\hat{\mathbb{E}}_j[X_{\tau_j}]=\underset{\tau\in \mathcal{T}_{j,N}}{ess\sup}\hat{\mathbb{E}}_j[X_\tau].
%	\end{displaymath}
%\end{description}

\begin{proposition}\label{os8}
	Assume that
	\begin{displaymath}
	\tau_{j}=\inf\{l\geq j: \tilde{V}_l^\infty=X_l\}
	\end{displaymath}
satisfies condition \eqref{re14}. Then we have% is finite quasi-surely (i.e. $c(\tau_j>N)\rightarrow 0$, as $N\rightarrow\infty$)
	\begin{description}
		\item[(i)] $\tau_{j}\in \mathcal{T}_j$ and $\tilde{V}^\infty_{n\wedge \tau_j}\in L_G^{1^*_*}(\Omega_n)$, for each $n\in\mathbb{N}$;
		\item[(ii)] $\{\tilde{V}^\infty_{n\wedge \tau_j},n=j,j+1,\cdots\}$ is a $G$-martingale;
		\item[(iii)] For any $j\in \mathbb{N}$,
		\begin{displaymath}
		\tilde{V}^\infty_j=\hat{\mathbb{E}}_j[X_{\tau_j}]=\underset{\tau\in \mathcal{T}_{j}}{ess\sup}\textrm{ }\hat{\mathbb{E}}_j[X_\tau].
		\end{displaymath}
	\end{description}
\end{proposition}

\begin{proof}
	%For any $N\in \mathbb{N}$ and $j\leq N$, set
	%\begin{displaymath}
	%	\tau_j^N =\inf\{l\geq j:V_l^N=X_l\}.
	%\end{displaymath}
	%Noting that for any $j\leq M\leq N$ and $l\leq M$, we have $X_l\leq V^M_l \leq V^N_l$, which implies $\{V_l^N=X_l\}\subset \{V_l^M=X_l\}$. Therefore, we can conclude that $\tau_{j}^N\uparrow \tau_j$, as $N\rightarrow\infty$.
	\textbf{(i)}	Noting that for any $k\in \mathbb{N}$ and $N\geq k$, $\tilde{V}_k^N\geq X_k$ and $\tilde{V}_k^N\uparrow \tilde{V}_k^\infty$ as $N\rightarrow\infty$, then we have $\{\tilde{V}_k^\infty=X_k\}=\cap_{N\geq k}^\infty\{\tilde{V}_k^N=X_k\}$, which implies
	\begin{displaymath}
	I_{\{\tilde{V}_k^\infty=X_k\}}=\inf_{N\geq k}I_{\{\tilde{V}_k^N=X_k\}}.
	\end{displaymath}
	By the proof of Theorem \ref{os1}, we have $I_{\{\tilde{V}_k^N=X_k\}}\in L_G^{1^*}(\Omega_k)$. Applying Proposition \ref{os20}, we can check that $I_{\{\tilde{V}_k^\infty=X_k\}}\in  L_G^{1^*}(\Omega_k)$. Since
	\begin{displaymath}
	I_{\{\tau_j\leq n\}}=\max_{j\leq k\leq n}I_{\{\tilde{V}_k^\infty=X_k\}},
	\end{displaymath}
	it follows that $I_{\{\tau_j\leq n\}}\in  L_G^{1^*}(\Omega_n)$. Without loss of generality, we may assume $X_n\geq 0$ for any $n\in \mathbb{N}$. It is easy to check that
	\begin{displaymath}
	\tilde{V}_{n\wedge \tau_j}^\infty=\sum_{k=j}^{n-1}(\tilde{V}_k^\infty-\tilde{V}_{k+1}^\infty)I_{\{\tau_j\leq k\}}+\tilde{V}_n^\infty.
	\end{displaymath}
	Since $I_{\{\tau_j\leq k\}}\in  L_G^{1^*}(\Omega_k)$, there exists a bounded sequence $\{\xi_n^{j,k}\}_{n=1}^\infty\subset L_G^1(\Omega_k)$ such that $\xi_n^{j,k}\downarrow I_{\{\tau_j\leq k\}}$. Note that
	\begin{align*}
	-\tilde{V}^N_{k+1} \xi_n^{j,k}\downarrow -\tilde{V}_{k+1}^\infty\xi_n^{j,k}, &\textrm{ as } N\rightarrow \infty,\\
	-\tilde{V}_{k+1}^\infty\xi_n^{j,k}\uparrow -\tilde{V}_{k+1}^\infty I_{\{\tau_j\leq k\}}, &\textrm{ as } n\rightarrow \infty,\\
	\tilde{V}^N_{k} \xi_n^{j,k}\downarrow\tilde{V}_k^N I_{\{\tau_j\leq k\}}, &\textrm{ as } n\rightarrow \infty,\\
	\tilde{V}_k^N I_{\{\tau_j\leq k\}}\uparrow \tilde{V}_k^\infty I_{\{\tau_j\leq k\}}, &\textrm{ as } N\rightarrow \infty.
	\end{align*}
	It follows that $-\tilde{V}_{k+1}^\infty I_{\{\tau_j\leq k\}}\in L_G^{1_*^*}(\Omega_{k+1})$ and $\tilde{V}_k^\infty I_{\{\tau_j\leq k\}}\in L_G^{1_*^*}(\Omega_{k})$. Hence, $\tilde{V}^{\infty}_{n\wedge \tau_j}\in L_G^{1_*^*}(\Omega_{n})$.
	
	\textbf{(ii)} Note that
	\begin{equation}\label{re16}
	\tilde{V}^\infty_{(n+1)\wedge \tau_j}-\tilde{V}^\infty_{n\wedge \tau_j}= I_{\{\tau_j\geq n+1\}}(\tilde{V}^\infty_{n+1}-\tilde{V}^\infty_n)=I_{\{\tau_j\leq n\}^c}(\tilde{V}^\infty_{n+1}-\hat{\mathbb{E}}_n[\tilde{V}^\infty_{n+1}]).
	\end{equation}
	Since $\{\tau_j\leq n\}\in \mathcal{B}(\Omega_n)$ and $\tilde{V}^\infty_{n+1}$, $-\hat{\mathbb{E}}_n[\tilde{V}^\infty_{n+1}]\in L_G^{*1}(\Omega_{n+1})$, applying Lemma \ref{os7} and \eqref{re15}, we have
	\begin{equation}\label{re17}
	\hat{\mathbb{E}}_n[I_{\{\tau_j\leq n\}^c}(\tilde{V}^\infty_{n+1}-\hat{\mathbb{E}}_n[\tilde{V}^\infty_{n+1}])]=I_{\{\tau_j\leq n\}^c}\hat{\mathbb{E}}_n[\tilde{V}^\infty_{n+1}-\hat{\mathbb{E}}_n[\tilde{V}^\infty_{n+1}]]=0.
	\end{equation}
	By a similar analysis as Step (i), we can get that $-\tilde{V}^\infty_{n\wedge \tau_j}\in L_G^{1_*^*}(\Omega_{n})$. The above two equalities implies that
	\begin{displaymath}
	0=\hat{\mathbb{E}}_n[\tilde{V}^\infty_{(n+1)\wedge \tau_j}-\tilde{V}^\infty_{n\wedge \tau_j}]=\hat{\mathbb{E}}_n[\tilde{V}^\infty_{(n+1)\wedge \tau_j}]-\tilde{V}^\infty_{n\wedge \tau_j},
	\end{displaymath}
	which shows that $\{\tilde{V}^\infty_{n\wedge \tau_j},n=j,j+1,\cdots\}$ is a $G$-martingale.
	
	\textbf{(iii)}
	First, we claim that there exists some $1<p<\beta$ such that
	\begin{displaymath}
	\hat{\mathbb{E}}[\sup_{n\in \mathbb{N}}|\tilde{V}_n^\infty|^p]<\infty.
	\end{displaymath}
	By Theorem \ref{os1}, we have $\tilde{V}_j^N=\hat{\mathbb{E}}_j[X_{\tau_j^N}]$, where $\tau_j^N=\inf\{l\geq j:\tilde{V}_l^N=X_l\}$. It is easy to check that $|\tilde{V}_j^N|\leq \hat{\mathbb{E}}_j[\sup_{n\in \mathbb{N}}|X_n|]$ and
	\begin{displaymath}
	\hat{\mathbb{E}}[\sup_{1\leq j\leq N}|\tilde{V}_j^N|^p]\leq \hat{\mathbb{E}}[\sup_{1\leq j\leq N}\hat{\mathbb{E}}_j[\sup_{n\in \mathbb{N}}|X_n|^p]].
	\end{displaymath}
	Since $\hat{\mathbb{E}}[\sup_{n\in\mathbb{N}}|X_n|^\beta]<\infty$, by Theorem 3.4 in \cite{S11}, there exists a constant $C$ independent of $N$ such that $\hat{\mathbb{E}}[\sup_{1\leq j\leq N}|\tilde{V}_j^N|^p]\leq C$. By monotone convergence theorem, we have
	\begin{displaymath}
	\hat{\mathbb{E}}[\sup_{n\in \mathbb{N}}|\tilde{V}_n^\infty|^p]=\lim_{N\rightarrow \infty}\hat{\mathbb{E}}[\sup_{1\leq j\leq N}|\tilde{V}_j^N|^p]\leq C
	\end{displaymath}
	We then show that
	\begin{displaymath}
	\tilde{V}^\infty_j=\hat{\mathbb{E}}_j[\tilde{V}^\infty_{\tau_j}]=\hat{\mathbb{E}}_j[X_{\tau_j}].
	\end{displaymath}
	Indeed, by Step (ii), we have for any $n\geq j$
	\begin{equation}\label{re18}
	\tilde{V}^\infty_j=\hat{\mathbb{E}}_j[\tilde{V}^\infty_{n\wedge \tau_j}].
	\end{equation}
	For any $\varepsilon>0$, there exists some $N>0$ such that, for any $n\geq N$, $c(\tau_j>n)\leq \varepsilon$. It is easy to check that
	\begin{displaymath}
	\hat{\mathbb{E}}[|\hat{\mathbb{E}}_j[\tilde{V}^\infty_{n\wedge \tau_j}]-\hat{\mathbb{E}}_j[\tilde{V}^\infty_{\tau_j}]|]\leq \hat{\mathbb{E}}[|\tilde{V}^\infty_{n\wedge \tau_j}-\tilde{V}^\infty_{\tau_j}|]\leq \hat{\mathbb{E}}[2\sup_{n\in \mathbb{N}}|\tilde{V}_n^\infty|I_{\{\tau_j>n\}}]\leq C(\hat{\mathbb{E}}[\sup_{n\in \mathbb{N}}|\tilde{V}_n^\infty|^p])^{1/p}\varepsilon^{1/q},
	\end{displaymath}
	where $\frac{1}{p}+\frac{1}{q}=1$. First, letting $n\rightarrow\infty$, since $\varepsilon$ is arbitrarily small, \eqref{re18} yields that $\tilde{V}^\infty_j=\hat{\mathbb{E}}_j[\tilde{V}^\infty_{ \tau_j}]$. In the following, we show that for any $\tau\in \mathcal{T}_j$, $\tilde{V}_j^\infty\geq \hat{\mathbb{E}}_j[X_\tau]$. For any $\tau\in \mathcal{T}_j$ and $\varepsilon>0$, there exists some $N$ such that  $c(\tau>N)\leq \varepsilon$. We can calculate that
	\begin{displaymath}
	\hat{\mathbb{E}}[|X_\tau-X_{\tau \wedge N}|]\leq \hat{\mathbb{E}}[2\sup_{n\in\mathbb{N}}|X_n|I_{\{\tau>N\}}]\leq C\varepsilon^{\frac{\beta-1}{\beta}}.
	\end{displaymath}
	It follows that
	\begin{displaymath}
	\hat{\mathbb{E}}_j[X_\tau]=\lim_{N\rightarrow \infty}\hat{\mathbb{E}}[X_{\tau\wedge N}].
	\end{displaymath}
	By a similar analysis as the proof of Theorem \ref{os1}, we have for each $N\geq j$, $\tilde{V}_j^\infty\geq \hat{\mathbb{E}}_j[X_{\tau\wedge N}]$. Letting $N\rightarrow\infty$, we deduce that $\tilde{V}_j^\infty\geq \hat{\mathbb{E}}_j[X_\tau]$. This completes the proof.
	%and $\tilde{V}_n^N$ is continuous in $\omega$, for any $N\geq n$, then $\tilde{V}_n^\infty$ is lower semicontinuous in $\omega$.-\tilde{V}_{k+1}^\infty =\lim_{n\rightarrow \infty}\lim_{N\rightarrow \infty},
\end{proof}

\begin{remark}
	For each fixed $N\in\mathbb{N}$, we may define the sequence $\{\underline{V}^N_n,n=0,1,\cdots,N\}$ recursively: Let $\underline{V}^N_N=X_N$ and
	\begin{displaymath}
	\underline{V}_n^N=\min\{X_n,\hat{\mathbb{E}}_n[\underline{V}^N_{n+1}]\}, \ n\leq N-1.
	\end{displaymath}
	We can check that for any $n\leq N\leq M$, $\underline{V}_n^N\leq \underline{V}_n^M$. It is natural to define
	\begin{displaymath}
	\underline{V}^\infty_n=\lim_{N\geq n,N\rightarrow \infty}\underline{V}_n^N.
	\end{displaymath}
	Then silimar results still hold for the sequence $\{\underline{V}^\infty,n\in\mathbb{N}\}$. More precisely, set \begin{displaymath}
	\tau_{j}=\inf\{l\geq j: \underline{V}_l^\infty=X_l\}.
	\end{displaymath}
	Assume that $\tau_j$ is finite quasi-surely (i.e. $c(\tau_j>N)\rightarrow 0$, as $N\rightarrow\infty$). Then we have
	\begin{description}
		\item[(i)] $\tau_{j}\in \mathcal{T}_j$ and $\underline{V}^\infty_{n\wedge \tau_j}\in L_G^{*1}(\Omega_n)$, for each $n\in\mathbb{N}$;
		\item[(ii)] $\{\underline{V}^\infty_{n\wedge \tau_j},n=j,j+1,\cdots\}$ is a $G$-martingale;
		\item[(iii)] For any $j\in \mathbb{N}$,
		\begin{displaymath}
		\underline{V}^\infty_j=\hat{\mathbb{E}}_j[X_{\tau_j}]=\underset{\tau\in \mathcal{T}_{j}}{ess\inf
		}\textrm{ }\hat{\mathbb{E}}_j[X_\tau],
		\end{displaymath}
		\item[(iv)] The sequence $\{\underline{V}^\infty_n,n\in\mathbb{N}\}$ is the largest $G$-submartingale dominated by the process $\{X_n,n\in\mathbb{N}\}$.
	\end{description}
\end{remark}

\section{Optimal stopping in continuous time}

\subsection{Finite time horizon case}
In this subsection, we provide the relation between the valuc function of the optimal stopping problem and the solution of reflected $G$-BSDE. For simplicity, assume the time horizon is $[0,1]$. We need to consider the following payoff process $\{X_t\}_{t\in[0,1]}$.%\in S_G^\beta(0,T)$ and $X$ is bounded from below.
\begin{assumption}\label{a3}
	The payoff process $\{X_t\}_{t\in[0,1]}\in S_G^\beta(0,1)$ (for the definition, we may refer to Appendix B), where $\beta>1$.% and $X$ is bounded from below
\end{assumption}

Denote by $\mathcal{T}^\infty_{s,t}$ the collection of all $G$-stopping times $\tau$ such that $s\leq \tau\leq t$ and by $\mathcal{T}^n_{s,t}$ the collection of all $G$-stopping times taking values in $\mathcal{I}_n$ such that $s\leq \tau\leq t$, where $0\leq s<t$ and $\mathcal{I}_n=\{k/2^n,k=0,1,\cdots,2^n\}$. Set
\begin{equation}\label{re9}
V_0=\sup_{\tau\in\mathcal{T}^\infty_{0,1}}\hat{\mathbb{E}}[X_\tau].
\end{equation}
For each $n\in\mathbb{N}$, we define the following sequence $\{V^n_{t^n_k},k=0,1,\cdots,2^n\}$ backwardly: Let $V^n_1=X_1$ and
\begin{displaymath}
V^n_{t^n_k}=\max(X_{t^n_k},\hat{\mathbb{E}}_{t^n_k}[V^n_{t^n_{k+1}}]),\ k=0,1,\cdots,2^n-1,
\end{displaymath}
where $t^n_k=k/2^n$. By Theorem \ref{os1}, for any $n\in\mathbb{N}$ and $k=0,1,\cdots,2^n$, we have
\begin{displaymath}
V^n_{t^n_k}=\underset{\tau\in \mathcal{T}^n_{0,1},\tau\geq t^n_k}{ess\sup}\hat{\mathbb{E}}_{t^n_k}[X_\tau].
\end{displaymath}
It is easy to check that for any $n\in \mathbb{N}$ and $k=0,1,\cdots,2^n$, $V^n_{t^n_k}\leq V^{n+1}_{t^n_k}$. Then we may define
\begin{equation}\label{re10}
V_{t^n_k}^\infty=\lim_{m\geq n, m\rightarrow\infty}V_{t^n_k}^m.
\end{equation}

\begin{proposition}\label{os4}
	Let $\mathcal{I}=\cup_{n}\mathcal{I}_n$. For each $t\in\mathcal{I}$, we have $V^\infty_t\in L_G^{*1}(\Omega_t)$. Morevoer, the sequence $\{V_t^\infty, t\in \mathcal{I}\}$ is the smallest $G$-supermartingale dominating the process $\{X_t,t\in \mathcal{I}\}$.
\end{proposition}

\begin{proof}
	Let $t_k^n, t_l^m\in\mathcal{I}$ and $t_k^n<t_l^m$, where $m,n\in \mathbb{N}$. It is easy to check that
	\begin{align*}
	\hat{\mathbb{E}}_{t_k^n}[V_{t_l^m}^\infty]&=\hat{\mathbb{E}}_{t_k^n}[\lim_{M\geq m, M\rightarrow\infty}V^M_{t_l^m}]=\hat{\mathbb{E}}_{t_k^n}[\lim_{M\geq (m\vee n), M\rightarrow\infty}V^M_{t_l^m}]\\
	&=\lim_{M\geq (m\vee n), M\rightarrow\infty}\hat{\mathbb{E}}_{t_k^n}[V_{t^m_l}^M]\leq \lim_{M\geq (m\vee n), M\rightarrow\infty}V^M_{t_k^n}=V^\infty_{t_k^n}.
	\end{align*}
	Now let $\{U_t,t\in \mathcal{I}\}$ be a $G$-supermartingale dominating $\{X_t,t\in\mathcal{I}\}$. By Theorem \ref{os1}, we know that $\{V^n,t\in\mathcal{I}_n\}$ is the smallest $G$-supermartingale dominating $\{X_t,t\in\mathcal{I}\}$. Therefore, for any $m\geq n$, we have $U_{t_k^n}\geq V^m_{t_k^n}$. Letting $m\rightarrow\infty$, we get that $U_{t_k^n}\geq V^\infty_{t_k^n}$, which completes the proof.
\end{proof}

\begin{proposition}\label{os2}
	Assume that the payoff process $\{X_t\}_{t\in[0,1]}$ satisfies Assumption \ref{a3}. Then we have
	\begin{displaymath}
	V_0=V_0^\infty.
	\end{displaymath}	
\end{proposition}

\begin{proof}
	Note that for each $n$,
	\begin{displaymath}
	V_0^n=\sup_{\tau\in\mathcal{T}^n_{0,1}}\hat{\mathbb{E}}[X_\tau].
	\end{displaymath}
	Consequently, we have $V_0\geq V_0^n$, $n\in \mathbb{N}$. Letting $n$ tends to infinity, we get $V_0\geq V_n^\infty$. Now we prove the inverse inequality. For each $\tau\in \mathcal{T}_{0,1}$ and $n\in\mathbb{N}$, set
	\begin{displaymath}
	\tau^n=\frac{1}{2^n}I_{\{0\leq \tau\leq \frac{1}{2^n}\}}+\sum_{k=2}^{2^n}\frac{k}{2^n}I_{\{\frac{k-1}{2^n}<\tau\leq \frac{k}{2^n}\}}.
	\end{displaymath}
	It is easy to check that $\tau^n\in \mathcal{T}^n_{0,1}$ and $|\tau^n-\tau|\leq 1/2^n$. Applying the continuity property of $X$ (see Lemma \ref{the3.7}), we have
	\begin{displaymath}
	\lim_{n\rightarrow \infty}\hat{\mathbb{E}}[|X_\tau-X_{\tau^n}|]=0.
	\end{displaymath}
	It follows that
	\begin{displaymath}
	\hat{\mathbb{E}}[X_\tau]=\lim_{n\rightarrow\infty}\hat{\mathbb{E}}[X_{\tau^n}]\leq \lim_{n\rightarrow \infty} V_0^n=V_0^\infty.
	\end{displaymath}
	Since $\tau$ is arbitrarily choosen, we deduce that $V_0\leq V_0^\infty$.
\end{proof}

According to \cite{CR}, we know that the value function of the optimal stopping problem defined by $g$-expectation coincides with the solution of reflected BSDE with a lower obstacle. It is natural to conjecture that our value function \eqref{re9} defined by $G$-expectation corresponds to the solution of reflected BSDE driven by $G$-Brownian motion. The solution of reflected $G$-BSDE is a triple of processes $(Y,Z,L)$ such that the first component $Y$ lies above the obstacle process $X$ and the last component can be regarded as the force to push $Y$ upwards which behaves in a minimal way satisfying the martingale condition instead of the Skorohod condition. For more details, we may refer to the Appendix B.% More precisely, we have the following result.
\begin{theorem}\label{os3}
	Let $(Y,Z,L)$ be the solution of reflected $G$-BSDE with terminal value $X_1$, generator 0 and obstacle process $X$. Then we have $Y_0=V_0$.
\end{theorem}

\begin{proof}
	By Proposition \ref{1}, $Y$ is a $G$-supermartingale dominating the process $X$. By Theorem \ref{os1}, for all $n\in\mathbb{N}$, we have $Y_0\geq V_0^n$. Applying Proposition \ref{os2}, it follows that $Y_0\geq \lim_{n\rightarrow \infty}V_0^n=V_0$. We then show the inverse inequality. For each fixed $n\in \mathbb{N}$ and $\varepsilon>0$, set
	\begin{displaymath}
	\tau^n_\varepsilon=\inf\{t\in\mathcal{I}_n,t\geq 0:A_t\leq\varepsilon \},
	\end{displaymath}
	where $A_t=Y_t-X_t$. It is easy to check that $\tau^n_\varepsilon$ is a $G$-stopping time and it is decreasing in $n$ and $\varepsilon$. Furthermore
	\begin{displaymath}
	\tau_\varepsilon=\lim_{n\rightarrow \infty}\tau_\varepsilon^n=\inf\{t\in\mathcal{I},t\geq 0:A_t\leq\varepsilon \}.
	\end{displaymath}
	Let
	\begin{displaymath}
	\tau=\inf\{t\geq 0:A_t=0\}.
	\end{displaymath}
	By Proposition 7.7 in \cite{LPSH}, we have $Y_0=\hat{\mathbb{E}}[X_\tau]$. We claim that
	\begin{equation}\label{re11}
	\lim_{\varepsilon\rightarrow 0}\lim_{n\rightarrow \infty}\tau^n_\varepsilon=\lim_{\varepsilon\rightarrow 0}\tau_\varepsilon=\tau.
	\end{equation}
	We prove \eqref{re11} in two cases. Suppose that $\tau(\omega)=t\in \mathcal{I}$. Then there exists some $n\in \mathbb{N}$ such that $t\in \mathcal{I}_n$. For any $k\geq n$ and $\varepsilon>0$, we have $\tau^k_\varepsilon(\omega)\leq t$. For each fixed $m\in \mathbb{N}$, Note that $A_t(\omega)$ is continuous in $t$. Denote by $\varepsilon_m$ the minimum of $A_t(\omega)$ on the interval $[0,t-1/m]$. For any $\varepsilon<\varepsilon_m$ and $n\in\mathbb{N}$, we have $\tau_\varepsilon^n(\omega)>t-1/m$ and $\tau_\varepsilon(\omega)>t-1/m$. We conclude that for any $k\geq n$ and $\varepsilon<\varepsilon_m$,
	\begin{displaymath}
	t-\frac{1}{m}<\tau_\varepsilon^k(\omega)\leq t.
	\end{displaymath}
	First letting $k\rightarrow\infty$ and $\varepsilon\rightarrow0$ and then letting $m\rightarrow\infty$, we show that \eqref{re11} holds true when $\tau(\omega)\in \mathcal{I}$. If $\tau(\omega)=t\notin \mathcal{I}$, there exists a sequence $\{t_k\}\subset\mathcal{I}$ such that $t_k\downarrow t$. For any $\varepsilon<\varepsilon_m$, there exists a constant $K_m$ such that, for any $k\geq K_m$,
	\begin{displaymath}
	A_{t_k}(\omega)=|A_{t_k}(\omega)-A_t(\omega)|\leq \varepsilon.
	\end{displaymath}
	It follows that $\tau_\varepsilon(\omega)\leq t_k$. We deduce that for any $\varepsilon<\varepsilon_m$ and $k\geq K_m$,
	\begin{displaymath}
	t-\frac{1}{m}<\tau_\varepsilon(\omega)\leq t_k.
	\end{displaymath}
	First letting $k\rightarrow \infty$, then letting $\varepsilon\rightarrow0$ and finally letting $m\rightarrow\infty$, the above inequality yields that \eqref{re11} holds true. By the continuity property of $X$ (see Lemma \ref{the3.7}), we have
	\begin{displaymath}
	Y_0=\hat{\mathbb{E}}[X_\tau]=\lim_{\varepsilon\rightarrow 0}\lim_{n\rightarrow \infty}\hat{\mathbb{E}}[X_{\tau_\varepsilon^n}]\leq \lim_{\varepsilon\rightarrow 0}\lim_{n\rightarrow \infty}V_0^n=V_0.
	\end{displaymath}
\end{proof}

\begin{remark}\label{re0}
	By a similar proof of Proposition 7.7 in \cite{LPSH}, for any $t\in[0,1]$, we have $Y_t=\hat{\mathbb{E}}_t[X_{\tau_t}]$, where $(Y,Z,L)$ is the solution of reflected $G$-BSDE with data $(X_1,0,X)$ and
	\begin{displaymath}
		\tau_t=\inf\{s\geq t: Y_s-X_s=0\}.
	\end{displaymath}
	Modified the proof of Theorem \ref{os3} slightly, we obtain that $Y_t=V_t^\infty$, for any $t\in\mathcal{I}$.
\end{remark}

%\begin{remark}
%	For each $t\in\mathcal{I}$, we assume that the following essential supremum is well defined
%	\begin{displaymath}
%	V_t=\underset{\tau\in \mathcal{T}^\infty_{t,1}}{ess\sup}\hat{\mathbb{E}}_t[X_\tau].
%	\end{displaymath}
%	By a similar analysis as in the proof of Theorem \ref{os2}, we have
%	\begin{displaymath}
%	V_t=V_t^\infty, \textrm{ for all } t\in \mathcal{I}.
%	\end{displaymath}
%\end{remark}

With the help of the relation between value function of optimal stopping problem and the solution of reflected BSDE driven by $G$-Brownian motion, we may get the following representation theorem similar with the discrete time case.
\begin{theorem}\label{os17}
	For each $t\in\mathcal{I}$, we  have %assume that the following essential supremum is well defined
	\begin{displaymath}
	V_t^\infty=\underset{\tau\in \mathcal{T}^\infty_{t,1}}{ess\sup}\textrm{ }\hat{\mathbb{E}}_t[X_\tau].
	\end{displaymath}
\end{theorem}

\begin{proof}
	Since $t\in\mathcal{I}$, we assume that $t=\frac{k}{2^n}$ for some $n\in \mathbb{N}$. Then for any $\tau\in \mathcal{T}_{t,1}^\infty$ and $m\geq n$, set
	\begin{displaymath}
		\tau^m=\frac{2^{m-n}k+1}{2^m}I_{\{\frac{k}{2^n}\leq \tau\leq \frac{2^{m-n}k+1}{2^m}\}}+\sum_{i=2}^{2^m-2^{m-n}k}\frac{2^{m-n}k+i}{2^m}I_{\{\frac{2^{m-n}k+i-1}{2^m}<\tau\leq \frac{2^{m-n}k+i}{2^m}\}}.
	\end{displaymath}
	We can check that $\tau^m\in \mathcal{T}_{t,1}^m$ and $|\tau-\tau^m|\leq 1/2^m$. By the continuity property of $X$, it follows that
	\begin{displaymath}
		\hat{\mathbb{E}}_t[X_\tau]=\lim_{m\rightarrow \infty}\hat{\mathbb{E}}_t[X_{\tau^m}]\leq \lim_{m\rightarrow \infty}V_t^m=V_t^{\infty}.
	\end{displaymath}
	We now claim that if $\eta\geq \hat{\mathbb{E}}_t[X_\tau]$, for any $\tau\in \mathcal{T}_{t,1}^\infty$, then $\eta\geq V_t^\infty$. For each fixed $n\in\mathbb{N}$ and $\varepsilon>0$, set
	\begin{displaymath}
		\tau_t^{n,\varepsilon}=\inf\{s\in \mathcal{I}_n, s\geq t: A_s\leq \varepsilon\},
	\end{displaymath}
	where $\{A\}$ is the process defined in Theorem \ref{os3}. By a similar analysis, we have $\tau_t^{n,\varepsilon}\in \mathcal{T}_{t,1}^n$ and
	\begin{displaymath}
		\lim_{\varepsilon\rightarrow 0}\lim_{n\rightarrow \infty}\tau_t^{n,\varepsilon}=\tau_t.
	\end{displaymath}
	Applying Lemma \ref{the3.7} and Remark \ref{re0} yields that
	\begin{displaymath}
		V_t^\infty=\hat{\mathbb{E}}_t[X_{\tau_t}]=\lim_{\varepsilon\rightarrow 0}\lim_{n\rightarrow \infty}\hat{\mathbb{E}}_t[X_{\tau_t^{n,\varepsilon}}]\leq \eta.
	\end{displaymath}
    The proof is complete.
\end{proof}

\subsection{Infinite time horizon case}
In this subsection, we investigate the infinite time horizon case. The payoff process $X$ satisfies the following assumption.
\begin{assumption}
	$\{X_t,t\geq 0\}$ is bounded from below and for any $n\in\mathbb{N}$, $X\in S_G^\beta(0,n)$, where $\beta>1$. Furthermore, $\hat{\mathbb{E}}[\sup_{t\geq 0}|X_t|^\beta]<\infty$.
\end{assumption}

Set $t^n_k=\frac{k}{2^n}$ and $\mathcal{I}^\infty_n=\{t_k^n:k=0,1,\cdots\}$. For each fixed $n,N\in \mathbb{N}$, define the following sequence $\{V^{n,N}_{t_k^n}:k=0,1,\cdots,N\}$ backwardly:
\begin{displaymath}
	V^{n,N}_{t_N^n}=X_{t_N^n},\ \ V^{n,N}_{t_k^n}=\max\{X_{t_k^n},\hat{\mathbb{E}}_{t_k^n}[V^{n,N}_{t_{k+1}^n}]\}, k\leq N-1.
\end{displaymath}
It is easy to see that for any $k\leq N\leq M$, $V^{n,N}_{t_k^n}\leq V^{n,M}_{t_k^n}$. We define
\begin{displaymath}
	V^n_{t_k^n}=\lim_{N\geq k,N\rightarrow \infty}V^{n,N}_{t_k^n}.
\end{displaymath}
Then $V^n_{t_k^n}\in L^{*1}_G(\Omega_{t_k^n})$. If $n<m$, the for any $N\in \mathbb{N}$, we can easily check that
\begin{displaymath}
	V^{n,N}_{t_k^n}\leq V^{m,N2^{m-n}}_{t_{k2^{m-n}}^m}, k\leq N,
\end{displaymath}
which yields that $V^n_{t_k^n}\leq V^m_{t_{k2^{m-n}}^m}=V^m_{t_k^n}$. We define
\begin{displaymath}
	V_{t_k^n}=\lim_{m\geq n, m\rightarrow\infty}V^m_{t_{k2^{m-n}}^m}=\lim_{m\geq n, m\rightarrow\infty}V^m_{t_k^n}.
\end{displaymath}
Then $V^n_{t_k^n}\in L^{*1}_G(\Omega_{t_k^n})$ and $V$ satisfies the following property.

\begin{proposition}\label{os18}
	The sequence $\{V_t,t\in \mathcal{I}^\infty\}$ is the smallest $G$-supermartingale dominating $\{X_t,t\in \mathcal{I}^\infty\}$, where $\mathcal{I}^\infty=\cup_{n=1}^\infty\mathcal{I}_n^\infty$. Besides, we have
	\begin{equation}\label{re26}
		V_0=\sup_{\tau\in \mathcal{T}^\infty_0}\hat{\mathbb{E}}[X_\tau],
	\end{equation}
	where $\mathcal{T}^\infty_t$ is the collection of all $G$-stopping time taking values in $[t,\infty)$ and satisfying equation \eqref{re14}.
\end{proposition}

\begin{proof}
	 By Proposition \ref{os5} and \ref{os6}, we derive that $\{V^n_{t_k^n}, k\in \mathbb{N}\}$ is the smallest $G$-supermartingale dominating $\{X_{t_k^n},k\in \mathbb{N}\}$ and
	\begin{displaymath}
	V^n_0=\sup_{\tau\in \mathcal{T}^n_0}\hat{\mathbb{E}}[X_\tau],
	\end{displaymath}
	where $\mathcal{T}^n_t$ is the collection of all $G$-stopping time taking values in $\mathcal{I}^\infty_n$, greater or equal to $t$ and satisfying equation \eqref{re14}. It is easy to check that for any $t_k^n,t_l^m\in \mathcal{I}^\infty$ with $t_k^n<t_l^m$, we have
	\begin{displaymath}
		\hat{\mathbb{E}}_{t_k^n}[V_{t_l^m}]=\hat{\mathbb{E}}_{t_k^n}[\lim_{M\geq m, M\rightarrow\infty} V^{M}_{t_{l2^{M-m}}^M}]=\lim_{M\geq (m\vee n), M\rightarrow\infty}\hat{\mathbb{E}}_{t_k^n}[V_{t_{l2^{M-m}}^M}^M]\leq \lim_{M\geq (m\vee n), M\rightarrow\infty}V_{t_{k2^{M-n}}^M}^M=V_{t_k^n},
	\end{displaymath}
	which yields that $\{V_t,t\in \mathcal{I}^\infty\}$ is a $G$-supermartingale. If $\{U_t,t\in \mathcal{I}^\infty\}$ is another $G$-supermartingale dominating $\{X_t,t\in \mathcal{I}^\infty\}$, then for $t=t_k^n\in \mathcal{I}^\infty$ and $m\geq n$, it is easy to check that $U_{t_k^n}\geq V^m_{t_k^n}=V^m_{t_{k2^{m-n}}^m}$. It follows that
	\begin{displaymath}
		U_{t_k^n}\geq \lim_{m\geq n, m\rightarrow\infty} V^m_{t_k^n} =V_{t_k^n},
	\end{displaymath}
	which implies that $\{V_t,t\in \mathcal{I}^\infty\}$ is the smallest $G$-supermartingale dominating $\{X_t,t\in \mathcal{I}^\infty\}$. To prove equation \eqref{re26}, first note that $V_0=\lim_{n\rightarrow \infty}V^n_0\leq \sup_{\tau\in \mathcal{T}_0^\infty}\hat{\mathbb{E}}[X_\tau]$. On the other hand, for any $\tau\in \mathcal{T}^\infty_0$, there exists $\tau^n\in \mathcal{T}^n_0$ such that $\tau^n\rightarrow\tau$. Noting that $X$ is continuous and applying Lemma \ref{fatou}, we get
	\begin{displaymath}
		\hat{\mathbb{E}}[X_\tau]=\hat{\mathbb{E}}[\liminf_{n\rightarrow \infty}X_{\tau^n}]\leq \liminf_{n\rightarrow \infty}\hat{\mathbb{E}}[X_{\tau^n}]\leq \liminf_{n\rightarrow \infty}V_0^n=V_0.
	\end{displaymath}
	Since $\tau$ is chosen arbitrarily, the proof is complete.
\end{proof}

\begin{remark}
	In fact, $\{V_t,t\in\mathcal{I}^\infty\}$ can be defined by the following procedure. By equation \eqref{re10} and Proposition \ref{os4}, we can construct a sequence $\{V^{\infty,N}_t,t\in\mathcal{I}^\infty,t\leq N\}$ such that it is the smallest $G$-supermartingale dominating the process $\{X_t,t\in\mathcal{I}^\infty,t\leq N\}$. Besides, by Theorem \ref{os17}, for any $t\in\mathcal{I}^\infty$ with $t\leq N$, we have
	\begin{displaymath}
		V_t^{\infty,N}=\esssup_{\tau\in\mathcal{T}_t^\infty,\tau\leq N}\hat{\mathbb{E}}_t[X_\tau].
	\end{displaymath}
	It is easy to check that for $t\leq N\leq M$, $V_t^{\infty,N}\leq V_t^{\infty,M}$. For each $t=t_k^n\in\mathcal{I}^\infty$, We may define
	\begin{displaymath}
		\tilde{V}_t=\lim_{N\geq t, N\rightarrow\infty}V_t^{\infty,N}.
	\end{displaymath}
	We claim that $V_t=\tilde{V}_t$ for any $t\in \mathcal{I}^\infty$. It suffices to prove that $\{\tilde{V}_t,t\in \mathcal{I}^\infty\}$ is the smallest $G$-supermartingale dominating $\{X_t,t\in\mathcal{I}^\infty\}$. For any $s,t\in\mathcal{I}^\infty$ with $s\leq t$, we have
	\begin{displaymath}
		\hat{\mathbb{E}}_s[\tilde{V}_t]=\hat{\mathbb{E}}_s[\lim_{N\geq t,N\rightarrow \infty}V^{\infty,N}_t]=\lim_{N\geq t, N\rightarrow\infty}\hat{\mathbb{E}}_s[V^{\infty,N}_t]\leq \lim_{N\geq t, N\rightarrow\infty}V^{\infty,N}_s=\tilde{V}_s.
	\end{displaymath}
	Now suppose that $\{U_t,t\in \mathcal{I}^\infty\}$ is the a $G$-supermartingale dominating $\{X_t,t\in\mathcal{I}^\infty\}$, then we have $U_t\geq V_t^{\infty,N}$ for any $t\leq N$. Letting $N\rightarrow\infty$ yields that $U_t\geq \tilde{V}_t$.
	
	By a similar analysis as the proof of Proposition \ref{os8}, we derive that for any $\tau\in\mathcal{T}_t^\infty$ and $t\in \mathcal{I}^\infty$
	\begin{displaymath}
		\hat{\mathbb{E}}_t[X_\tau]=\lim_{N\geq t,N\rightarrow \infty}\hat{\mathbb{E}}_t[X_{\tau\wedge N}]\leq \lim_{N\geq t,N\rightarrow\infty}V_t^{\infty,N}=V_t.
	\end{displaymath}
	On the other hand, if there exists some $\eta\in \mathcal{L}(\Omega_t)$, such that $\eta\geq \hat{\mathbb{E}}_t[X_\tau]$ for any $\tau\in \mathcal{T}_t^\infty$ with some $t\in\mathcal{I}^\infty$, then $\eta\geq V_t^{\infty,N}$ for any $N\geq t$, which implies that $\eta\geq V_t$. By the definition of essential supremum, the above analysis shows that
	\begin{displaymath}
		V_t=\esssup_{\tau\in\mathcal{T}^\infty_t}\hat{\mathbb{E}}_t[X_\tau],
	\end{displaymath}
	for each $t\in\mathcal{I}^\infty$.
\end{remark}

\section{Markovian case}

In this section, we will present some results of optimal stopping under $G$-expectation when the payoff process is Markovian. More precisely, consider the payoff process $\{X^{t,\xi}\}$ generated by the following $G$-SDE:
\begin{equation}\label{re24}
	X_s^{t,\xi}=\xi+\int_t^s b(X_r^{t,\xi})dr+\int_t^s h(X_r^{t,\xi})d\langle B\rangle_r+\int_t^s \sigma(X_r^{t,\xi})dB_r,
\end{equation}
where $\xi\in L_G^p(\Omega_t)$, $p\geq 2$ and $b$, $h$, $\sigma:\mathbb{R}\rightarrow \mathbb{R}$ are deterministic functions satisfying the following:
\begin{description}
	\item[(H1)] There exists a constant $L>0$, such that for any $x,y\in \mathbb{R}$
	\begin{displaymath}
		|b(x)-b(y)|+|h(x)-h(y)|+|\sigma(x)-\sigma(y)|\leq L|x-y|.
	\end{displaymath}
\end{description}
Then we have the following estimates, which can be found in \cite{LPSH,P10}.
\begin{proposition}\label{the1.17}
	Let $\xi,\xi'\in L_G^p(\Omega_t)$ and $p\geq 2$. Then we have, for each $\delta\in[0,T-t]$,
	\begin{align*}
	\hat{\mathbb{E}}_t[\sup_{s\in[t,t+\delta]}|X_{s}^{t,\xi}-X_{s}^{t,\xi'}|^p]&\leq C|\xi-\xi'|^p,\\
	%\hat{E}_t[|X_{t+\delta}^{t,\xi}-X_{t+\delta}^{t,\xi'}|^p]&\leq C|\xi-\xi'|^p,\\
	\hat{\mathbb{E}}_t[|X_{t+\delta}^{t,\xi}|^p]&\leq C(1+|\xi|^p),\\
	\hat{\mathbb{E}}_t[\sup_{s\in[t,t+\delta]}|X_s^{t,\xi}-\xi|^p]&\leq C(1+|\xi|^p)\delta^{p/2},
	\end{align*}
	where the constant $C$ depends on $L,G,p$ and $T$.
\end{proposition}
For simplicity, set $X_s^x:=X_s^{0,x}$. By Lemma 4.1 in \cite{HJL}, we have the following Markov property.

\begin{lemma}\label{os9}
	For each given $\varphi\in C_{b,Lip}(\mathbb{R})$ and $s,t\geq 0$, we have
	\begin{displaymath}
		\hat{\mathbb{E}}_t[\varphi(X_{t+s}^x)]=\hat{\mathbb{E}}[\varphi(X_{s}^{y})]_{y=X_t^x}
	\end{displaymath}
\end{lemma}

\subsection{Discrete time case}
In this subsection, we first investigate the discrete time case. For a given function $f\in C_{b,Lip}(\mathbb{R})$ and $x\in \mathbb{R}$, consider the following optimal stopping problem
\begin{equation}\label{re19}
     F^N(x):=\sup_{\tau\in \mathcal{T}_{0,N}}\hat{\mathbb{E}}[f(X^x_\tau)].
\end{equation}

\begin{lemma}\label{os10}
	For each $N\in \mathbb{N}$, the function $F^N$ defined by equation \eqref{re19} is bounded and Lipschitz.
\end{lemma}

\begin{proof}
	Since $f\in C_{b,Lip}(\mathbb{R})$, $F^N$ is bounded. Besides, by Proposition \ref{the1.17}, we have
	\begin{displaymath}
		|F^N(x)-F^N(y)|\leq \sup_{\tau\in \mathcal{T}_{0,N}}\hat{\mathbb{E}}[|f(X_\tau^x)-f(X_\tau^y)|]\leq C\hat{\mathbb{E}}[\sup_{t\in[0,N]}|X_t^x-X_t^y|]\leq C|x-y|.
	\end{displaymath}
\end{proof}

Set $\tilde{X}^x_n=f(X_n^x)$. It is easy to check that $\{\tilde{X}^x_n,n\in \mathbb{N}\}$ satisfies Assumption \ref{a2}. Similar with Section 3, we define the following sequence $\{V^N_n,n=0,1,\cdots,N\}$ backwardly.  Let ${V}^N_N(x)=\tilde{X}^x_N$ and
\begin{displaymath}
{V}_n^N(x)=\max\{\tilde{X}^x_n,\hat{\mathbb{E}}_n[{V}^N_{n+1}(x)]\}, \ n\leq N-1.
\end{displaymath}
It is important to note that $V_0^N(x)=F^N(x)$. Moreover, we have the following identity
\begin{equation}\label{re20}
	V_n^N(x)=F^{N-n}(X_n^x), \textrm{ for } 0\leq n\leq N.
\end{equation}
This will be shown in the proof of the next theorem. Now we set
\begin{align*}
	C_n&=\{x\in \mathbb{R}: F^{N-n}(x)>f(x)\},\\
	D_n&=\{x\in \mathbb{R}: F^{N-n}(x)=f(x)\},
\end{align*}
for any $n=0,1,\cdots,N$. Then we define
\begin{displaymath}
	\tau_D^{N,x}=\inf\{0\leq n\leq N: X_n^x\in D_n\}.
\end{displaymath}
Since both $F^{N-n}$ and $f$ are Lipschitz continuous, then $D_n$ is a closed set which implies that $I_{\{X_n^x\in D_n\}}\in L_G^{1^*}(\Omega_n)$. Therefore, we may conclude that $\tau_D^{N,x}$ is a $G$-stopping time. Finally, for any $f\in C_{b,Lip}(\mathbb{R})$, define the following transition operator $T$:
\begin{displaymath}
	T f(x)=\hat{\mathbb{E}}[f(X_1^x)].
\end{displaymath}

\begin{theorem}\label{os11}
	Consider the optimal stopping time problem \eqref{re19}. Then for any $n=1,2,\cdots,N$, the value function $F^n$ satisfies the Wald-Bellman equations
	\begin{equation}\label{re1}
		F^n(x)=\max\{f(x),TF^{n-1}(x)\},
	\end{equation}
	where $F^0(x)=f(x)$. Furthermore, we have
	\begin{description}
		\item[(i)] $\tau_D^{N,x}$ is a $G$-stopping time and optimal in equation \eqref{re19};
		\item[(ii)] The sequence $\{F^{N-n}(X_n^x),n=0,1,\cdots,N\}$ is the smallest $G$-supermartingale which dominates $\{f(X_n^x), n=0,1,\cdots,N\}$ for each $x\in \mathbb{R}$;
		\item[(iii)] The stopped process $\{F^{N-n\wedge \tau_D^{N,x}}(X_{n\wedge\tau_D^{N,x}}^x),n=0,1,\cdots,N\}$ is a $G$-martingale for each $x\in \mathbb{R}$.
	\end{description}
\end{theorem}

\begin{proof}
	We claim that $\{V_0^n\}$ satisfies the Wald-Bellman equations. Indeed, it is easy to check that $V_0^0(x)=\tilde{X}_0^x=f(x)$ and
	\begin{displaymath}
		V_0^1(x)=\max\{\tilde{X}_0^x,\hat{\mathbb{E}}[V_1^1(x)]\}=\max\{f(x),\hat{\mathbb{E}}[f(X_1^x)]\}=\max\{f(x),TV_0^0(x)\}.
	%	V_1^2(x)&=\max\{f(X_1^x),\hat{\mathbb{E}}_1[f(X_2^x)]\}=\max\{f(X_1^x),\hat{\mathbb{E}}[f(X_1^y)]_{y=X_1^x}\}=V_0^1(X_1^x).
	\end{displaymath}
	 We assume that for any $n\leq k$, $V_0^n(x)=\max\{f(x),\hat{\mathbb{E}}[V_0^{n-1}(X_1^x)]\}$. We then calculate that
	 \begin{displaymath}
	 		V_k^{k+1}(x)=\max\{f(X_k^x),\hat{\mathbb{E}}_k[f(X_{k+1}^x)]\}=\max\{f(X_k^x),\hat{\mathbb{E}}[f(X_1^y)]_{y=X_k^x}\}=V_0^1(X_k^x),
	 \end{displaymath}
	 and
	 \begin{align*}
	 	V_{k-1}^{k+1}(x)&=\max\{f(X_{k-1}^x),\hat{\mathbb{E}}_{k-1}[V_k^{k+1}(x)]\}=\max\{f(X_{k-1}^x),\hat{\mathbb{E}}[V_0^{1}(X_k^x)]\}\\
	 	&=\max\{f(X_{k-1}^x),\hat{\mathbb{E}}[V_0^1(X_1^y)]_{y=X_{k-1}^x}\}=V_0^2(X_{k-1}^x).
	 \end{align*}
	 By the above procedure, we have
	 \begin{align*}
	 V_{1}^{k+1}(x)&=\max\{f(X_{1}^x),\hat{\mathbb{E}}_{1}[V_2^{k+1}(x)]\}=\max\{f(X_{1}^x),\hat{\mathbb{E}}[V_0^{k-1}(X_2^x)]\}\\
	 &=\max\{f(X_{1}^x),\hat{\mathbb{E}}[V_0^{k-1}(X_1^y)]_{y=X_{1}^x}\}=V_0^k(X_{1}^x),
	 \end{align*}
	 which yields that
	 \begin{displaymath}
	 	V_0^{k+1}(x)=\max\{f(x),\hat{\mathbb{E}}[V_1^{k+1}(x)]\}=\max\{f(x),\hat{\mathbb{E}}[V_0^{k}(X_1^x)]\}.
	 \end{displaymath}
	 Note that the above analysis also establishes that  for any $0\leq j\leq n\leq N$, $V_j^n(x)=V_0^{n-j}(X_j^x)$. Recall that $F^n(x)=V_0^n(x)$, $n=0,1,\cdots,N$, which implies that \eqref{re20} holds and $F^n$ satisfies the Wald-Bellman equation. Applying Theorem \ref{os1}, the conclusions (i)-(iii) hold.
\end{proof}

For the infinite time case, the value function is defined by
\begin{equation}\label{re21}
	F(x)=\sup_{\tau\in \mathcal{T}_0}\hat{\mathbb{E}}[f(X_\tau^x)],
\end{equation}
where $f\in C_{b,Lip}(\mathbb{R})$. Let $V_n^\infty(x)=\lim_{N\rightarrow\infty}V_n^N(x)$.  By Proposition \ref{os6}, we have
\begin{displaymath}
	F(x)=V_0^\infty(x)=\lim_{N\rightarrow\infty}F^N(x),
\end{displaymath}
which implies that $F$ is a bounded lower semicontinuous function. Then letting $N\rightarrow \infty$ in equation \eqref{re20}, it follows that
\begin{displaymath}
     V_n^\infty(x)=F(X_n^x), \textrm{ for } n\in \mathbb{N}.
\end{displaymath}

Set
\begin{align*}
C&=\{x\in \mathbb{R}: F(x)>f(x)\},\\
D&=\{x\in \mathbb{R}: F(x)=f(x)\}.
\end{align*}
Since $F$ is lower semicontinuous, $D$ is a closed subset of $\mathbb{R}$. Then we define
\begin{displaymath}
\tau_D^x=\inf\{n\geq 0: X_n^x\in D\}.
\end{displaymath}
Similar with the finite time case, $\tau_D^x$ is a $G$-stopping time.

\begin{definition}\label{d1}
      A measurable function $F:\mathbb{R}\rightarrow\mathbb{R}$ is said to be superharmonic, if for all $x\in \mathbb{R}$,
      \begin{displaymath}
      	TF(x)\leq F(x).
      \end{displaymath}
\end{definition}

\begin{remark}
	We should point out that there is an implicit assumption in the above definition that $F(X_1^x)\in\mathbb{L}^1(\Omega)$ for each $x\in \mathbb{R}$.
\end{remark}

\begin{lemma}\label{os12}
	Suppose that $F$ is lower semicontinuous and bounded from below (resp. upper semicontinuous and bounded from above). Then $F$ is superharmonic if and only if $\{F(X_n^x),n\in \mathbb{N}\}$ is a $G$-supermartingale for any $x\in \mathbb{R}$.
\end{lemma}

\begin{proof}
	Since $F$ is lower semicontinuous and bounded from below, there exists a sequence $\{F^m,m\in \mathbb{N}\}\subset C_{b,Lip}(\mathbb{R})$, such that $F^m\uparrow F$. For the ``if" part, suppose that $F$ is superharmonic. Note that $F^m(X_n^x)\in L_G^1(\Omega_n)$ and $F^m(X_n^x)\uparrow F(X_n^x)$ for any $m,n\in\mathbb{N}$ and $x\in \mathbb{R}$. We can calculate that
	\begin{equation}\label{re23}
	\begin{split}
	F(X_n^x)&\geq TF(X_n^x)=\hat{\mathbb{E}}[F(X_1^y)]_{y=X_n^x}=\lim_{m\rightarrow \infty}\hat{\mathbb{E}}[F^m(X_1^y)]_{y=X_n^x}\\
	&=\lim_{m\rightarrow \infty}\hat{\mathbb{E}}_n[F^m(X_{n+1}^x)]=\hat{\mathbb{E}}_n[F(X_{n+1}^x)],
	\end{split}
	\end{equation}
	for all $n\in \mathbb{N}$ and $x\in \mathbb{R}$, which implies $\{F(X_n^x),n\in \mathbb{N}\}$ is a $G$-supermartingale. For the ``only if" part, note that $F(X_n^x)\geq \hat{\mathbb{E}}_n[F(X_{n+1}^x)]$ holds for any $n\in \mathbb{N}$. Letting $n=0$ yields that $F$ is superharmonic.
\end{proof}

\begin{theorem}\label{os13}
	Consider the optimal stopping time problem \eqref{re21}. Then  the value function $F$ satisfies the Wald-Bellman equations
	\begin{equation}\label{re2}
	F(x)=\max\{f(x),TF(x)\},
	\end{equation}
	Furthermore, assume that for any $x\in\mathbb{R}$, $\tau_D^x$ satifies \eqref{re14}. Then we have
	\begin{description}
		\item[(i)] $\tau_D^x$ is a $G$-stopping time and optimal in equation \eqref{re21};
		\item[(ii)] The value function $F$ is the smallest superharmonic function which dominates $f$ on $\mathbb{R}$;%The sequence $\{F^{N-n}(X_n^x),n=0,1,\cdots,N\}$ is the smallest $G$-supermartingale which dominates $\{f(X_n^x), n=0,1,\cdots,N\}$ for each $x\in \mathbb{R}$;
		\item[(iii)] The stopped process $\{F(X_{n\wedge\tau_D^x}^x),n=0,1,\cdots,N\}$ is a $G$-martingale for each $x\in \mathbb{R}$.
	\end{description}
\end{theorem}

It is easy to check that for any $n=1,2,\cdots,N$, there exists a unique solution to the Wald-Bellman equation \eqref{re1}. However, when the time horizon is infinite, there may be many solutions to the Wald-Bellman equation \eqref{re2}. For example, if $f(x)\equiv c$, the any $F(x)=C\geq c$ solves this equation. In the following, we give a sufficient condition under which the solution to equation \eqref{re2} is unique.
\begin{theorem}\label{os14}
     Suppose that $G:\mathbb{R}\rightarrow\mathbb{R}$ is lower semicontinuous and bounded from below satisfying the Wald-Bellman equation
     \begin{displaymath}
     	G(x)=\max\{f(x), TG(x)\},
     \end{displaymath}
     for $x\in \mathbb{R}$. Furthermore, we assume that, for some $p>1$,
     \begin{displaymath}
     	\hat{\mathbb{E}}[\sup_{n\in \mathbb{N}}|G(X_n^x)|^p]<\infty,
     \end{displaymath}
     for any $x\in \mathbb{R}$. If the following ``boundary condition at infinity" holds,
     \begin{equation}\label{re22}
     	\limsup_{n\rightarrow\infty}G(X_n^x)=\limsup_{n\rightarrow\infty}F(X_n^x), \ \ \ q.s.
     \end{equation}
     for any $x\in \mathbb{R}$, then $G$ equals to the value function $F$.
\end{theorem}

\begin{proof}
Without loss of generality, we assume that $G\geq 0$. Since $G$ satisfy the Wald-Bellman equation, it is superharmonic and $G\geq f$. By Theorem \ref{os13}, we have $G\geq F$. In the following, we show the converse inequality. Define the random time
\begin{displaymath}
\tau_\varepsilon^x=\inf\{n\geq 0: G(X_n^x)\leq f(X_n^x)+\varepsilon\}=\inf\{n\geq 0: X_n^x\in D_\varepsilon\},
\end{displaymath}
where $\varepsilon>0$ and
\begin{displaymath}
D_\varepsilon=\{x\in\mathbb{R}: G(x)\leq f(x)+\varepsilon\}.
\end{displaymath}
It is a closed subset of $\mathbb{R}$ due to the lower semicontinuity of $G$. Therefore, $\tau_\varepsilon^x$ is a $G$-stopping time for any $\varepsilon>0$ and $x\in \mathbb{R}$. Besides, by \eqref{re22}, $\tau_\varepsilon^x$ satisfy condition \eqref{re14}. We claim that $\{G(X^x_{\tau_\varepsilon^x \wedge n}),n\in \mathbb{N}\}$ is a $G$-martingale for all $x\in \mathbb{R}$. By a similar analysis in the proof of Theorem \ref{os8}, we have $G(X^x_{\tau_\varepsilon^x \wedge n})\in L_G^{1^*_*}(\Omega_n)$, for each $n\in \mathbb{N}$. Note that $I_{\{\tau_\varepsilon^x\leq n-1\}}\in L_G^{1^*}(\Omega_{n-1})$. It follows that $I_{\{\tau_\varepsilon^x\geq n\}}\in L_G^{1^*_*}(\Omega_{n-1})$ and $G(X^x_{\tau_\varepsilon^x\wedge (n-1)})I_{\{\tau_\varepsilon^x\leq n-1\}}\in L_G^{1^*_*}(\Omega_{n-1})$. We can calculate that, for each $n\geq 1$ and $x\in \mathbb{R}$,
\begin{align*}
\hat{\mathbb{E}}_{n-1}[G(X^x_{\tau_\varepsilon^x \wedge n})]&=\hat{\mathbb{E}}_{n-1}[G(X^x_{n})I_{\{\tau_\varepsilon^x\geq n\}}]+G(X^x_{\tau_\varepsilon^x\wedge (n-1)})I_{\{\tau_\varepsilon^x\leq n-1\}}\\
&=\hat{\mathbb{E}}_{n-1}[G(X_n^x)]I_{\{\tau_\varepsilon^x\geq n\}}+G(X^x_{\tau_\varepsilon^x})I_{\{\tau_\varepsilon^x\leq n-1\}}\\
&=TG(X_{n-1}^x)I_{\{\tau_\varepsilon^x\geq n\}}+G(X^x_{\tau_\varepsilon^x})I_{\{\tau_\varepsilon^x\leq n-1\}}\\
&=G(X_{n-1}^x)I_{\{\tau_\varepsilon^x\geq n\}}+G(X^x_{\tau_\varepsilon^x})I_{\{\tau_\varepsilon^x\leq n-1\}}\\
&=G(X^x_{\tau_\varepsilon^x \wedge (n-1)}),
\end{align*}
where in the third equality we use equation \eqref{re23}. Therefore, we have
\begin{displaymath}
\hat{\mathbb{E}}[G(X^x_{\tau_\varepsilon^x \wedge n})]=G(x),
\end{displaymath}
for all $n\geq 0$ and $x\in\mathbb{R}$. A similar analysis as \eqref{re} shows that
\begin{displaymath}
\hat{\mathbb{E}}[G(X^x_{\tau_\varepsilon^x })]=\lim_{n\rightarrow \infty}\hat{\mathbb{E}}[G(X^x_{\tau_\varepsilon^x \wedge n})]=G(x),
\end{displaymath}
for all $x\in \mathbb{R}$. Recalling \eqref{re21} and the definition of $\tau_\varepsilon^x$, we derive that
\begin{align*}
F(x)\geq \hat{\mathbb{E}}[f(X^x_{\tau_\varepsilon^x})]\geq \hat{\mathbb{E}}[G(X_{\tau_\varepsilon^x}^x)]-\varepsilon=G(x)-\varepsilon.
\end{align*}
Since $\varepsilon$ can be arbitrarily small, we get $F\geq G$, which completes the proof.
\end{proof}

\subsection{Continuous time case}

In this section, we investigate the optimal stopping problem in the continuous time case when the payoff process is Markovian satisfying equation \eqref{re24}. Similar with Definition \ref{d1}, a basic concept is the following:

\begin{definition}
	A measurable function $f:\mathbb{R}\rightarrow \mathbb{R}$ is called excessive (w.r.t $X_t$), if
	\begin{displaymath}
	f(x)\geq \hat{\mathbb{E}}[f(X_t^x)] \textrm{ for all } t\geq 0, x\in \mathbb{R}.
	\end{displaymath}
\end{definition}

\begin{definition}
	A measurable function $f:\mathbb{R}\rightarrow \mathbb{R}$ is called superharmonic (w.r.t $X_t$), if
	\begin{displaymath}
	f(x)\geq \hat{\mathbb{E}}[f(X_\tau^x)]
	\end{displaymath}
	for all $G$-stopping time $\tau$ satisfying equation \eqref{re14} and all $x\in\mathbb{R}$.
\end{definition}

\begin{remark}
	It is importiant to note that in the above definitions, there is an implicite assumption that $f(X_\tau^x)\in \mathbb{L}^1(\Omega)$ for each $G$-stopping time $\tau$ and $x\in\mathbb{R}$.
\end{remark}

It is easy to check that a superharmonic function is  excessive. The following proposition shows that the converse is true for some typical $f$.

\begin{proposition}\label{os15}
	Suppose that $f$ is bounded and lower semicontinuous. If $f$ is excessive, then it is also superharmonic.
\end{proposition}

\begin{proof}
	We first prove this result for $f\in C_{b,Lip}(\mathbb{R})$. Without loss of generality, we assume that $f\geq 0$.
	
	Step 1. Suppose that $\tau$ is a discrete $G$-stopping time of the following form:
	\begin{displaymath}
	\tau=\sum_{i=0}^{n} t_i I_{\{\tau=t_i\}}.
	\end{displaymath}
	By a similar analysis as the proof of Theorem \ref{os1}, we have $f(X_\tau^x)\in L_G^{*1}(\Omega_{t_n})$ and
	\begin{align*}
	&\hat{\mathbb{E}}[f(X_\tau^x)]=\hat{\mathbb{E}}[\sum_{i=0}^n f(X_{t_i}^x)I_{\{\tau=t_i\}}]\\
	=&\hat{\mathbb{E}}[\sum_{i=0}^{n-1} f(X_{t_i}^x)I_{\{\tau=t_i\}}+\hat{\mathbb{E}}_{t_{n-1}}[f(X_{t_n}^x)I_{\{\tau=t_n\}}]]\\
	=&\hat{\mathbb{E}}[\sum_{i=0}^{n-1} f(X_{t_i}^x)I_{\{\tau=t_i\}}+\hat{\mathbb{E}}_{t_{n-1}}[f(X_{t_n}^x)]I_{\{\tau=t_n\}}]\\
	=&\hat{\mathbb{E}}[\sum_{i=0}^{n-1} f(X_{t_i}^x)I_{\{\tau=t_i\}}+\hat{\mathbb{E}}[f(X_{t_n-t_{n-1}}^y)]_{y=X_{t_{n-1}}^x}I_{\{\tau=t_n\}}]\\
	\leq &\hat{\mathbb{E}}[\sum_{i=0}^{n-2} f(X_{t_i}^x)I_{\{\tau=t_i\}}+f(X_{t_{n-1}}^x)I_{\{\tau\geq t_{n-1}\}}]\leq \cdots\leq f(x),
	\end{align*}
	where we use the Markov property in the forth equality.
	
	Step 2. If $\tau$ is bounded and continuous, there exists a sequence of discrete $G$-stopping time $\{\tau^n\}_{n=1}^\infty$ such that $|\tau-\tau^n|\leq 1/2^n$. Applying the continuity of $f(X^x)$, similar with the proof of Proposition \ref{os2}, we have
	\begin{displaymath}
	\hat{\mathbb{E}}[f(X_\tau^x)]=\lim_{n\rightarrow \infty}\hat{\mathbb{E}}[f(X_{\tau^n}^x)]\leq f(x).
	\end{displaymath}
	
	Step 3. If $\tau$ satisfy equation \eqref{re14}, by a similar analysis as the proof of Proposition \ref{os6}, it follows that
	\begin{displaymath}
	\hat{\mathbb{E}}[f(X_\tau^x)]=\lim_{N\rightarrow\infty}\hat{\mathbb{E}}[f(X_{\tau\wedge N}^x)]\leq f(x).
	\end{displaymath}
	
	Now we show the result still hold for $f$ which is bounded and lower semicontinuous. We can choose a sequence $\{f_n\}_{n=1}^\infty\subset C_{b,Lip}(\mathbb{R})$ such that $f_n\uparrow f$. By the proof of Theorem \ref{os7}, we derive that $f(X_\tau^x)\in L_G^{1^*_*}(\Omega_{t_n})$ for each discrete $G$-stopping time $\tau$ with the form in Step 1. Besides, since $f_n(X_t^x)\uparrow f(X_t^x)$ for any $t\geq 0$ and $x\in \mathbb{R}$, we have
	\begin{equation}\label{re25}
	\hat{\mathbb{E}}_t[f(X_{t+s}^x)]=\lim_{n\rightarrow \infty}\hat{\mathbb{E}}_t[f_n(X_{t+s}^x)]=\lim_{n\rightarrow \infty}\hat{\mathbb{E}}[f_n(X_s^y)]_{y=X_t^x}=\hat{\mathbb{E}}[f(X_s^y)]_{y=X_t^x}.
	\end{equation}
	Then the proof of Step 1 can be extended to the case where $f$ is lower semicontinuous and bounded from below.
	If $\tau$ is continuous and bounded and $\tau^n$ is choosen as Step 2, noting that $f$ is lower semicontinuous and applying Fatou's Lemma, it is easy to check that
	\begin{displaymath}
	\hat{\mathbb{E}}[f(X_\tau^x)]\leq \hat{\mathbb{E}}[\liminf_{n\rightarrow\infty}f(X_{\tau^n}^x)]\leq \liminf_{n\rightarrow\infty}\hat{\mathbb{E}}[f(X_{\tau^n}^x)]\leq f(x).
	\end{displaymath}
	Repeating the proof of Step 3, we finally get the desired result.
\end{proof}

\begin{proposition}\label{os16}
	Suppose that $f$ is bounded and lower semicontinuous. Then $f$ is excessive if and only if $\{f(X_t^x)\}$ is a supermartingale for any $x\in\mathbb{R}$.
\end{proposition}

\begin{proof}
	If $\{f(X_t^x)\}$ is a supermartingale, it follows that
	\begin{displaymath}
	\hat{\mathbb{E}}[f(X_t^x)]\leq f(X_0^x)=f(x),
	\end{displaymath}
	which implies that $f$ is excessive. The other direction can be shown easily by using equation \eqref{re25}. The proof is complete.
\end{proof}

For any given bounded and Lipschitz continuous function $g$, by the following iterative procedure, we may construct the smallest superharmonic function dominating $g$.
\begin{proposition}\label{os26}
	Let $g$ be  bounded and Lipschitz continuous and define the following sequence: $g_0(x)=g(x)$,
	\begin{displaymath}
	g_n(x)=\sup_{t\in S_n}\hat{\mathbb{E}}[g_{n-1}(X_t^x)], \ \ n=1,2,\cdots,
	\end{displaymath}
	where $S_n=\{k\cdot 2^{-n}:0\leq k\leq 4^n\}$. Then $g_n\uparrow \bar{g}$ and $\bar{g}$ is the smallest superharmonic function dominating $g$.
\end{proposition}

\begin{proof}
	It is obvious that $\{g_n\}$ is bounded and increasing. Furthermore, by Proposition \ref{the1.17}, we have
	\begin{displaymath}
	|g_1(x)-g_1(y)|\leq \sup_{t\in S_1}\hat{\mathbb{E}}[|g(X_t^x)-g(X_t^y)|]\leq C\sup_{t\in S_1}\hat{\mathbb{E}}[|X_t^x-X_t^y|]\leq C|x-y|.
	\end{displaymath}
	By induction, we derive that $g_n$ is continuous. Define $\bar{g}(x)=\lim_{n\rightarrow \infty}g_n(x)$. Then $\bar{g}$ is bounded and lower semicontinuous. We claim that $\bar{g}$ is excessive. Indeed, we can show that
	\begin{displaymath}
	\bar{g}(x)\geq g_n(x)=\hat{\mathbb{E}}[g_{n-1}(X_t^x)], \textrm{ for any } t\in S_n, n\geq 1.
	\end{displaymath}
	Letting $n\rightarrow \infty$, it follows that
	\begin{displaymath}
	\bar{g}\geq \hat{\mathbb{E}}[\bar{g}(X_t^x)], \textrm{ for any } t\in S=\cup_{n=1}^\infty S_n.
	\end{displaymath}
	If $t\notin S$, there exists $\{t_n\}_{n=1}^\infty\subset S$ such that $t_n\rightarrow t$. By Fatou's Lemma and noting that $\bar{g}$ is lower semicontinuous, we have
	\begin{displaymath}
	\hat{\mathbb{E}}[\bar{g}(X_t^x)]\leq \hat{\mathbb{E}}[\liminf_{n\rightarrow\infty}\bar{g}(X_{t_n}^x)]\leq \liminf_{n\rightarrow\infty}\hat{\mathbb{E}}[\bar{g}(X_{t_n}^x)]\leq \bar{g}(x).
	\end{displaymath}
	The above two inequalities imply $\bar{g}$ is excessive. By Proposition \ref{os15}, $\bar{g}$ is superharmonic. If $f$ is a superharmonic function and $f\geq g$, by induction, it is easy to check that $f\geq g_n$ for any $n=1,2,\cdots$. Letting $n$ tend to infinity, we finally get the desired result.
\end{proof}

\begin{theorem}\label{os27}
	Let $g$ be a bounded and Lipschitz function. Define
	\begin{displaymath}
	V(x)=\sup_{\tau\in \mathcal{T}}\hat{\mathbb{E}}[g(X_\tau^x)],
	\end{displaymath}
	where $\mathcal{T}$ is the collection of all $G$-stopping times satisfying equation \eqref{re14}. Then $V$ is the smallest superharmonic function dominating $g$.
\end{theorem}

\begin{proof}
	Set $\mathcal{I}_n^\infty=\{0,\frac{1}{2^n},\cdots,\frac{k}{2^n},\cdots\}$ for any $n\geq 1$ and denote by $\mathcal{T}_n^\infty$ the set of all $G$-stopping times taking values in $\mathcal{I}^\infty_n$ and satisfying equation \eqref{re14}. Consider the following optimal stopping problem:
	\begin{displaymath}
	V^n(x)=\sup_{\tau\in \mathcal{T}_n^\infty}\hat{\mathbb{E}}[g(X_\tau^x)].
	\end{displaymath}
	It is easy to see that $V\geq V^n$, for any $n\geq 1$. On the other hand, for any $\tau\in \mathcal{T}$, there exists a sequence of $G$-stopping time $\{\tau^n\}$ such that $\tau^n\in \mathcal{T}_n^\infty$ and $\tau^n\rightarrow \tau$. Applying Fatou's Lemma, we have
	\begin{displaymath}
	\hat{\mathbb{E}}[g(X_\tau^x)]=\hat{\mathbb{E}}[\liminf_{n\rightarrow\infty} g(X_{\tau^n}^x)]\leq \liminf_{n\rightarrow\infty}\hat{\mathbb{E}}[g(X_{\tau^n}^x)]\leq \liminf_{n\rightarrow\infty}V^n(x),
	\end{displaymath}
	which implies that $V^n\uparrow V$ and $V$ is bounded and lower semicontinuous. By Lemma \ref{os12} and Theorem \ref{os13}, we know that $\{V^n(X_t^x),t\in \mathcal{I}_n^\infty\}$ is a $G$-supermartingale for each $x\in\mathbb{R}$. If $t\in \mathcal{I}^\infty=\cup_{n=1}^\infty\mathcal{I}_n^\infty$, without loss of generality, we assume that $t=\frac{k}{2^m}$. By simple calculation, we have
	\begin{displaymath}
	\hat{\mathbb{E}}[V(X_t^x)]=\lim_{n\rightarrow \infty,n\geq m}\hat{\mathbb{E}}[V^n(X_t^x)]\leq \lim_{n\rightarrow \infty,n\geq m}V^n(x)=V(x).
	\end{displaymath}
	If $t\notin \mathcal{I}^\infty$, there exists $\{t_n\}\subset \mathcal{I}^\infty$ such that $t_n\rightarrow t$. Noting that $V$ is lower semicontinuous, it follows that
	\begin{displaymath}
	\hat{\mathbb{E}}[V(X_t^x)]\leq \hat{\mathbb{E}}[\liminf_{n\rightarrow\infty} V(X_{t_n}^x)]\leq \liminf_{n\rightarrow\infty}\hat{\mathbb{E}}[V(X_{t_n}^x)]\leq V(x).
	\end{displaymath}
	We derive that $V$ is an excessive function. Applying Proposition \ref{os15}, $V$ is also superharmonic. For any superharmonic function $f\geq g$, it is easy to check that
	\begin{displaymath}
	V(x)=\sup_{\tau\in \mathcal{T}}\hat{\mathbb{E}}[g(X_\tau^x)]\leq \sup_{\tau\in \mathcal{T}}\hat{\mathbb{E}}[f(X_\tau^x)]\leq f(x),
	\end{displaymath}
	which yields that $V$ is the smallest superharmonic function dominating $g$.
\end{proof}

\section*{Conclusion}
In this paper, we use the $G$-stochastic analysis to study optimal stopping problem under Knightian uncertainty for both the discrete time case and the continuous time case, and the time horizon can be finite or infinite. In order to solve this problem, we first introduce a new kind of random times, called $G$-stopping times, such that the conditional $G$-expectation is well defined for each payoff process $X$ stopped at some $G$-stopping time $\tau$. Besides, since the multiple priors $\mathcal{P}$ represented the $G$-expectation is non-dominated, we need to define the essential supremum in the $\mathcal{P}$-quasi-surely sense.

For the discrete time case, when the time horizon is finite, we apply the method of backward induction to define value function $V$ and then to show that $V$ is the Snell envelope of $X$ and the first hitting time is an optimal stopping time. By taking the time horizon $N$ goes to infinity, we get similar results with the finite time case. By the refinement of the time interval, we can study the continuous time case and then establish the relation between the value function and the solution of reflected $G$-BSDE. Therefore, it helps to get numerical approximation for the value function as well as the solution of reflected $G$-BSDE. In a Markovian setting, for the discrete time case, we show that the value function is the solution to the Wald-Bellman equation. For the continuous time case, similar with the classical result, the value function coincides with the smallest superharmonic function dominating the payoff function.\\

\noindent \textbf{Acknowledgements} \  \ The auther would like to thank Frank Riedel for helpful comments and suggestions. Financial support from the German Research Foundation (DFG) via CRC 1283 is gratefully acknowledged.

\appendix
\renewcommand\thesection{Appendix A}
\section{ }
\renewcommand\thesection{A}

We now introduce some basic notations and results of $G$-expectation. As is shown in \cite{HP13}, we can extend the conditional $G$-expectation to space $\bar{L}_G^{1^*_*}(\Omega)$ and it satisfies the following property.
\begin{proposition}(\cite{HP13})\label{os19}
	For each $X\in \bar{L}_G^{1^*_*}(\Omega)$, we have, for each $P\in\mathcal{P}$,
	\[\hat{\mathbb{E}}_t[X]={\esssup_{Q\in\mathcal{P}(t,P)}}^P E^Q[X|\mathcal{F}_t],\ \   P\textrm{-a.s.},\]
	where $\mathcal{P}(t,P)=\{Q\in P: E_Q[X]=E_P[X], \forall X\in L_{ip}(\Omega_t)\}$.
\end{proposition}

Since the conditional expectation can be well-defined on the space $\bar{L}_G^{1^*_*}(\Omega)$, we can modify the definition of $G$-martingale (-sub, -supmartingale) slightly.
\begin{definition}
	A process $M=\{M_t\}_{t\in[0,T]}$ is called a $G$-martingale (-sub, -supermartingale, resp.), if for each $t\in[0,T]$, $M_t\in \bar{L}_G^{1^*_*}(\Omega_t)$ and $\hat{\mathbb{E}}_s[M_t]=M_s$ $(\geq, \leq, \textrm{resp.})$, for any $0\leq s\leq t\leq T$.
\end{definition}

The extended conditional $G$-expectation shares many properties with the classical conditional expectation except the linearity. More precisely, we have the following proposition.
\begin{proposition}(\cite{HP13})\label{os20}
	We have
	\begin{description}
		\item[(1)] $X,Y\in \bar{L}_G^{1^*_*}(\Omega)$, $X\leq Y$$\Rightarrow$$\hat{\mathbb{E}}_t[X]\leq \hat{\mathbb{E}}[Y]$;
		\item[(2)] $X\in \bar{L}_G^{1^*_*}(\Omega_t)$, $Y\in \bar{L}_G^{1^*_*}(\Omega)$$\Rightarrow$$\hat{\mathbb{E}}_t[X+Y]=X+\hat{\mathbb{E}}_t[Y]$;
		\item[(3)] $X,Y\in \bar{L}_G^{1^*_*}(\Omega)$$\Rightarrow$$\hat{\mathbb{E}}_t[X+Y]\leq \hat{\mathbb{E}}_t[X]+\hat{\mathbb{E}}_t[Y]$;
		\item[(4)]$X\in \bar{L}_G^{1^*_*}(\Omega_t)$ is bounded, $X\geq 0$, $Y\in \bar{L}_G^{1^*_*}(\Omega)$, $Y\geq 0$, $\lim_{n\rightarrow\infty}\hat{\mathbb{E}}[YI_{\{Y\geq n\}}]=0$$\Rightarrow$$\hat{\mathbb{E}}_t[XY]=X\hat{\mathbb{E}}_t[Y]$;
		\item[(5)] $X\in \bar{L}_G^{1^*_*}(\Omega)$$\Rightarrow$$\hat{\mathbb{E}}_s[\hat{\mathbb{E}}_t[X]]=\hat{\mathbb{E}}_{s\wedge t}[X]$;
		\item[(6)] $\{X_n\}_{n=1}^\infty\subset L_G^{1^*}(\Omega)(L_G^{1^*_*}(\Omega))$, $X_n\downarrow X(X_n\uparrow X)$ q.s.$\Rightarrow$ $X\in L_G^{1^*}(\Omega)(L_G^{1^*_*}(\Omega))$ and $\hat{\mathbb{E}}_t[X]=\lim_{n\rightarrow\infty}\hat{\mathbb{E}}_t[X_n]$
	\end{description}
\end{proposition}

Now we give some examples of random variables belong to the above spaces.
\begin{proposition}(\cite{HP13})\label{os21}
	We have
	\begin{description}
		\item[(1)] Let $X$ be a bounded upper (resp. lower) semicontinuous function on $\Omega$. Then $X\in L_G^{1^*}(\Omega)$ (resp. $X\in L_G^{*1}(\Omega)$);
		\item[(2)] Let $X\in L_G^1(\Omega,\mathbb{R}^n)$ and let $f$ be a bounded upper (resp. lower) semicontinuous function on $\mathbb{R}^n$. Then $f(X)\in L_G^{1^*}(\Omega)$ (resp. $f(X)\in L_G^{*1}(\Omega)$);
	\end{description}
\end{proposition}

\begin{remark}\label{o8}
	Let $X\in L_G^1(\Omega)$, $a\in \mathbb{R}$. Then by the above proposition, we have $I_{\{X\leq a\}}$, $I_{\{X\geq a\}}$, $I_{\{X=a\}}\in L_G^{1^*}(\Omega)$.
\end{remark}

\renewcommand\thesection{Appendix B}
\section{ }
\renewcommand\thesection{B}

In this section, we recalled some basic results about reflected BSDE driven by $G$-Brownian motion.  More details can be found in \cite{L, LPS, P07a,P08a,P10}.
\begin{definition}
\label{def2.6} (i) Let $M_{G}^{0}(0,T)$ be the collection of processes in the
following form: for a given partition $\{t_{0},\cdot\cdot\cdot,t_{N}\}=\pi
_{T}$ of $[0,T]$,
\[
\eta_{t}(\omega)=\sum_{j=0}^{N-1}\xi_{j}(\omega)\mathbf{1}_{[t_{j},t_{j+1})}(t),
\]
where $\xi_{i}\in L_{ip}(\Omega_{t_{i}})$, $i=0,1,2,\cdot\cdot\cdot,N-1$. For each
$p\geq1$ and $\eta\in M_G^0(0,T)$, let $\|\eta\|_{H_G^p}:=\{\hat{\mathbb{E}}[(\int_0^T|\eta_s|^2ds)^{p/2}]\}^{1/p}$, $\Vert\eta\Vert_{M_{G}^{p}}:=(\mathbb{\hat{E}}[\int_{0}^{T}|\eta_{s}|^{p}ds])^{1/p}$ and denote by $H_G^p(0,T)$,  $M_{G}^{p}(0,T)$ the completion
of $M_{G}^{0}(0,T)$ under the norm $\|\cdot\|_{H_G^p}$, $\|\cdot\|_{M_G^p}$, respectively.

(ii) Let $S_G^0(0,T)=\{h(t,B_{t_1\wedge t}, \ldots,B_{t_n\wedge t}):t_1,\ldots,t_n\in[0,T],h\in C_{b,Lip}(\mathbb{R}^{n+1})\}$. For $p\geq 1$ and $\eta\in S_G^0(0,T)$, set $\|\eta\|_{S_G^p}=\{\hat{\mathbb{E}}[\sup_{t\in[0,T]}|\eta_t|^p]\}^{1/p}$. Denote by $S_G^p(0,T)$ the completion of $S_G^0(0,T)$ under the norm $\|\cdot\|_{S_G^p}$.
\end{definition}

 We have the following continuity property for any $Y\in S_G^p(0,T)$ with $p>1$.

\begin{lemma}[\cite{LPS}]\label{the3.7}
For $Y\in S_G^p(0,T)$ with $p>1$, we have, by setting $Y_s:=Y_T$ for $s>T$,
\begin{displaymath}
F(Y):=\limsup_{\varepsilon\rightarrow0}(\hat{\mathbb{E}}[\sup_{t\in[0,T]}\sup_{s\in[t,t+\varepsilon]}|Y_t-Y_s|^p])^\frac{1}{p}=0.
\end{displaymath}
\end{lemma}

The parameters of reflected $G$-BSDE consist of the following three parts: the generator $f$, the obstacle process $\{X_t\}_{t\in[0,T]}$ and the terminal value $\xi$, where $f$ is the map
\begin{displaymath}
f(t,\omega,y,z):[0,T]\times\Omega_T\times\mathbb{R}^2\rightarrow\mathbb{R}.
\end{displaymath}

We will make the following assumptions: There exists some $\beta>2$ such that
\begin{description}
\item[(H1)] for any $y,z$, $f(\cdot,\cdot,y,z)\in M_G^\beta(0,T)$;
\item[(H2)] $|f(t,\omega,y,z)-f(t,\omega,y',z')|\leq L(|y-y'|+|z-z'|)$ for some $L>0$;
\item[(H3)] $\xi\in L_G^\beta(\Omega_T)$ and $\xi\geq S_T$, $q.s.$;
\item[(H4)] There exists a constant $c$ such that $\{X_t\}_{t\in[0,T]}\in S_G^\beta(0,T)$ and $X_t\leq c$, for each $t\in[0,T]$;
\item[(H4')] $\{X_t\}_{t\in[0,T]}$ has the following form
\begin{displaymath}
X_t=X_0+\int_0^t b(s)ds+\int_0^t l(s)d\langle B\rangle_s+\int_0^t \sigma(s)dB_s,
\end{displaymath}
where $\{b(t)\}_{t\in[0,T]}$, $\{l(t)\}_{t\in[0,T]}$ belong to $M_G^\beta(0,T)$ and $\{\sigma(t)\}_{t\in[0,T]}$ belongs to $H_G^\beta(0,T)$.
\end{description}

Let us now introduce the reflected $G$-BSDE with a lower obstacle. A triple of processes $(Y,Z,L)$ is called a solution of reflected $G$-BSDE with a lower obstacle if for some $1<\alpha\leq \beta$ the following properties hold:
\begin{description}
\item[(a)]$(Y,Z,L)\in\mathcal{S}_G^{\alpha}(0,T)$ and $Y_t\geq X_t$, $0\leq t\leq T$;
\item[(b)]$Y_t=\xi+\int_t^T f(s,Y_s,Z_s)ds
-\int_t^T Z_s dB_s+(L_T-L_t)$;
\item[(c)]$\{-\int_0^t (Y_s-X_s)dL_s\}_{t\in[0,T]}$ is a nonincreasing $G$-martingale.
\end{description}
Here, we denote by $\mathcal{S}_G^{\alpha}(0,T)$ the collection of processes $(Y,Z,L)$ such that $Y\in S_G^{\alpha}(0,T)$, $Z\in H_G^{\alpha}(0,T)$, $L$ is a continuous nondecreasing process with $L_0=0$ and $L\in S_G^\alpha(0,T)$.

\begin{theorem}\label{the1.14}
Suppose that $\xi$, $f$ satisfy \textsc{(H1)}--\textsc{(H3)} and $X$ satisfies \textsc{(H4)} or \textsc{(H4')}. Then, the reflected $G$-BSDE with data $(\xi,f,X)$ has a unique solution $(Y,Z,L)$. Moreover, for any $2\leq \alpha<\beta$ we have $Y\in S^\alpha_G(0,T)$, $Z\in H_G^\alpha(0,T)$ and $L\in S_G^{\alpha}(0,T)$.
%Let $\xi\in L_G^{\beta}(\Omega_T)$ with $\xi\geq S_T$, $f$ satisfies (H1) and (H2) and the obstacle process $(S_t)\in S_G^\beta(0,T)$ with $\beta>2$ is either bounded from above or of the following form
%\begin{displaymath}
%S_t=S_0+\int_0^t b(s)ds+\int_0^t l(s)d\langle B\rangle_s+\int_0^t \sigma(s)dB_s,
%\end{displaymath}
%where $\{b(t)\}$, $\{l(t)\}$, $\{\sigma(t)\}\in M_G^\beta(0,T)$ with $\beta>2$. Then there exists a unique solution of the reflected $G$-BSDE (3.1).
\end{theorem}

\begin{proposition}\label{1}
	Let $f=0$. Suppose that $\xi$ satisfies (H3) and $X$ satisfies (H4) or (H4'). Then the first argument $Y$ of the solution to the reflected $G$-BSDE with data $(\xi,0,X)$ is a $G$-supermartingale dominating the process $X$.
\end{proposition}

\end{document}